\documentclass[hidelinks,onefignum,onetabnum]{siamart251104}
\usepackage{amsmath}
\DeclareMathOperator*{\argmin}{argmin}
\usepackage{graphicx}
\usepackage{subfig}
\usepackage{amssymb}
\usepackage{algorithm,algpseudocode,amsopn}

\algrenewcommand\algorithmicrequire{\textbf{Input:}}
\algrenewcommand\algorithmicensure{\textbf{Output:}}

\newsiamthm{assumption}{Assumption}
\newtheorem{prop}{Proposition}
\newtheorem{remark}{Remark}
\newcommand{\ind}{\,\mbox{d}}
\newcommand{\rev}[1]{{ #1}}

\headers{Low-rank-assisted inverse scattering}{S. Meng}

\title{Low-rank-assisted  inverse medium scattering: Lipschitz stability and ensemble Kalman filter
}

\author{Shixu Meng\thanks{Department of Mathematical Sciences, The University of Texas at Dallas, Richardson, TX 75080, USA.}
  (\email{smeng@utdallas.edu}).}

\usepackage{amsopn}

\begin{document}

\maketitle

\begin{abstract}
In this work we study the theoretical Lipschitz stability and propose a low-rank-assisted numerical method for the inverse medium scattering beyond the Born region. The proposed low-rank structure is based on the disk prolate spheroidal wave functions, which are eigenfunctions of both the Born forward operator and  a Sturm-Liouville differential operator. We obtain Lipschitz stability for unknowns in a low-rank space in the fully nonlinear case and characterize the explicit Lipschitz constant in the linearized region. We further propose an ensemble Kalman filter to iteratively update the unknown in the proposed low-rank space whose dimension is intrinsically determined by the wave number. Moreover the ensembles are sampled  according to a novel trace class covariance operator motivated by the connection between the proposed low-rank space and the Sturm-Liouville differential operator. Finally numerical examples are provided to illustrate the feasibility of the proposed method.

\end{abstract}

\begin{keywords}
Inverse scattering, Lipschitz stability, low-rank structure, ensemble Kalman filter, generalized prolate spheroidal wave function
\end{keywords}

\begin{MSCcodes}
78A46,65N21,35R30
\end{MSCcodes}

\section{Introduction}
The inverse scattering problem plays an important role in a wide range of physical and engineering applications, such as ocean acoustics, seismic imaging, medical diagnosis, and  non-destructive  evaluation. The goal of inverse scattering is to image the hidden or internal features using wave measurements. It is a challenging problem because inverse scattering is intrinsically nonlinear and ill-posed, and the measurement data are inevitably corrupted by noise. We refer to the monographs    \cite{cakoni2014qualitative,cakoni2016inverse,colton2012inverse,kirsch2008factorization}  for an introduction to the field of inverse scattering. Motivated by recent data-driven machine learning methods \cite{chen2026data,desai2025neural,khoo2019switchnet,zhou2025recovery}, Lipschitz stability in finite dimensional space \cite{alberti2022infinite,bourgeois2013remark}, low-rank structures \cite{meng23data,zhou2024exploring}, and data assimilation approaches \cite{furuya2022inverse,iglesias2013ensemble,nakamura2015inverse,parzer2022convergence}, this work proposes a low-rank-assisted approach for the inverse medium scattering problem with far field data. We investigate theoretical Lipschitz stability in a low-rank space and propose a low-rank-assisted ensemble Kalman filter to numerically solve the inverse medium scattering problem.

With infinite dimensional measurement data, it was shown in \cite{bourgeois2013remark} that Lipschitz stability holds for unknowns in a finite dimensional space for general inverse problems.  Recently,  Lipschitz  stability was proved for both unknowns and measurement data in a finite dimensional space  \cite{alberti2022infinite} for general inverse problems. In general, the Lipschitz constant blows up as the dimension of the finite dimensional space goes to infinity and it is interesting to explore an appropriate finite dimensional space  for each specific inverse problem. We are thus motivated to investigate a suitable finite dimensional space (i.e. low-rank space) for the inverse medium scattering problem. In this work we propose to use a low-rank space based on the  disk prolate spheroidal wave functions (disk PSWFs); this choice is natural since the disk PSWFs are exactly the eigenfunctions of the Born (or linearized) forward operator. More importantly, the disk PSWFs  make it possible to  characterize  the  approximation capability of the low-rank space since they are eigenfunctions of a Sturm-Liouville differential operator at the same time (i.e., dual property). In the proposed low-rank space, we apply the theoretical tools in \cite{alberti2022infinite,bourgeois2013remark} to obtain Lipschitz stability in the fully nonlinear case and characterize the explicit Lipschitz constant in the Born region.  We also point out that exploring low-rank-assisted approaches is also motivated by recent data-driven machine learning methods \cite{chen2026data,desai2025neural,khoo2019switchnet,zhou2025recovery}, as machine learning also exploits relevant low-dimensional features.

Motivated by derivative-free methods such as the inverse Born series \cite{moskow2008convergence} and its application to inverse scattering with two coefficients \cite{cakoni2025recovery}, in this work we investigate another derivative-free method, namely the ensemble Kalman filter, to numerically solve the inverse scattering problem with the help of the above-mentioned low-rank structure.
There is a vast literature on the ensemble Kalman filter and we limit our survey to the ones close to our study. The idea of the ensemble Kalman filter was proposed in \cite{evensen1994sequential} for geophysical data assimilation.  In a broader context, the  ensemble Kalman filter is related to Bayesian approaches for inverse problems, cf. \cite{kaipio2005statistical}. Much attention has been attracted since the work of the ensemble Kalman filter for linear inverse problems \cite{iglesias2013ensemble}.  For linear problems, the  ensemble Kalman filter can be related to the Tikhonov-Phillips regularization in connection with the three-dimensional and four-dimensional variational data assimilation (i.e., 3D-VAR and 4D-VAR), cf. \cite{nakamura2015inverse} and \cite{parzer2022convergence}. Recently, data assimilation and Kalman filter was applied to the inverse medium scattering problem \cite{furuya2022inverse} based on the explicit use of the Fr\'echet derivative. Different from these studies, in this work we solve the inverse  medium  scattering numerically via the ensemble Kalman filter, a derivative-free method which only implicitly uses the idea of linearization but avoids the explicit calculation of the Fr\'echet derivative. Such a formulation of the ensemble Kalman filter is expected to be applicable to a wide range of inverse scattering problems in complex media when evaluating the derivatives becomes infeasible due to problem reformulation or when the underlying partial differential equation does not perfectly represent the real-world system.  Particularly in this work, we first introduce a forward map which maps the real and imaginary parts of the contrast in $L^2(B;\mathbb{R})\oplus L^2(B;\mathbb{R})$ to the real and imaginary parts of the processed far field datum in $L^2(B;\mathbb{R})\oplus L^2(B;\mathbb{R})$, and then heuristically interpret the ensemble Kalman filter for the nonlinear inverse scattering problem using its connection to the classical Tikhonov-Phillips regularization.  It is worth pointing out that the ensembles are only sampled in the proposed low-rank space whose dimension is  mathematically determined by the wave number, unlike other heuristic ways of choosing the dimension of low-rank spaces; as such, the ensemble size is only on the order of the intrinsic dimension of the low-rank space. Moreover, another contribution is the introduction of a novel covariance operator using the connection between the proposed low-rank structure and a Sturm-Liouville differential operator.

We now put these contributions in perspective.  Our first contribution is to theoretically demonstrate that the abstract Lipschitz stability theory of \cite{bourgeois2013remark,alberti2022infinite} applies to the processed far-field data, and then justify the use of disk PSWF space  through an explicit Lipschitz constant in the Born regime and approximation rates obtained from the dual property. Our second contribution is an algorithmic derivative-free ensemble Kalman filter implemented in the same low-rank space, with a covariance operator built from the Sturm-Liouville operator; in contrast to the Kalman filter approach of \cite{furuya2022inverse}, which relies on the Fr\'echet derivative of the forward map, the proposed method avoids derivatives altogether and iterates in a space whose dimension is set by the wave number. 

The  remainder  of the paper is organized as follows. In Section \ref{section: model classical}, we introduce the classical mathematical model of the inverse medium scattering problem and a reformulation motivated by the  reciprocity relation. The proposed low-rank space will be based on the disk PSWFs, whose necessary preliminaries are discussed in Section \ref{section: linearized low-rank structure}. We obtain in Section \ref{section: stability} the Lipschitz stability in low-rank spaces and characterize the explicit Lipschitz constant in the Born region.
We discuss the ensemble Kalman filter for the fully nonlinear inverse scattering problem in Section \ref{section: EnKF} and propose a novel low-rank-assisted ensemble Kalman filter. Finally, numerical experiments are provided in Section \ref{section: numerics} to illustrate the feasibility of the proposed method.

\section{Mathematical formulation of the inverse scattering problem} \label{section: model classical}
Given a wave number $k>0$, we first introduce the model of the direct scattering problem due to a plane wave
\begin{equation*} %
e^{i k x \cdot \hat{\theta}}, \quad \hat{\theta} \in \mathbb{S} :=\{ x \in \mathbb{R}^2: |x|=1\},
\end{equation*}
where $\hat{\theta}$ is the direction of propagation.
Let $\Omega \subset \mathbb{R}^2$ be an open and bounded set with
Lipschitz boundary $\partial \Omega$ such that $\mathbb{R}^2 \backslash \overline{\Omega}$  is connected. In this work we assume that  $\overline{\Omega} \subset \overline{B}$, where $B:=B(0,1)$ denotes the unit disk (the case $\Omega=B$ is allowed); otherwise a scaling can be applied to reformulate the problem. Given a possibly complex-valued contrast of the medium $q(x) \in L^\infty(\mathbb{R}^2)$ where $\mbox{Re}(1+q) > 0$, $\mbox{Im}(q) \ge 0$ on $\Omega$,  and $q=0$ on $\mathbb{R}^2 \backslash \overline{\Omega}$,  the direct scattering problem  is to find  the  total wave field $u^t(x;\hat{\theta};k) = e^{i k x \cdot \hat{\theta} } + u^s(x;\hat{\theta};k)$ belonging to $H^1_{loc}(\mathbb{R}^2)$  such that
\begin{eqnarray}
\Delta_x u^t(x;\hat{\theta};k) + k^2 \left(1+q(x)\right) u^t(x;\hat{\theta};k) =  0   &\mbox{ in }&  \mathbb{R}^2,   \label{medium us eqn1}\\
\lim_{r:=|x|\to \infty} r^{\frac{1}{2}}  \big( \frac{\partial u^s(x;\hat{\theta};k)}{\partial r} -ik u^s(x;\hat{\theta};k)\big) =0.  \label{medium us eqn2}
\end{eqnarray}
The scattered wave field is $u^s(\cdot;\hat{\theta};k)$. A scattered wave is called radiating   if it satisfies  the last equation, i.e., the Sommerfeld radiation condition, uniformly for all directions. The contrast $q$ is related to physical quantities such as the electric permittivity and magnetic permeability for polarized electromagnetic scattering, cf. \cite{cakoni2014qualitative}.
The above scattering problem \eqref{medium us eqn1}--\eqref{medium us eqn2} can be solved via the more general problem where one looks for a radiating solution $u^s \in H^1_{loc}(\mathbb{R}^2)$   to
\begin{equation} \label{medium us eqn1+2}
\Delta u^s + k^2(1+q) u^s = -k^2 q f,
\end{equation}
where $f \in L^\infty(\mathbb{R}^2)$. Setting $f(x)=e^{ikx\cdot \hat{\theta}}$ in \eqref{medium us eqn1+2} recovers \eqref{medium us eqn1}--\eqref{medium us eqn2}.

It is known that there exists a unique radiating solution to \eqref{medium us eqn1+2}, cf. \cite{colton2012inverse} and \cite{kirsch2017remarks}. For example, the solution can be obtained  with the help of the Lippmann-Schwinger integral equation,
\begin{eqnarray*}
u^s(x) - k^2 \int_\Omega \Phi(x,y)q(y)  u^s(y)  \ind y= k^2 \int_\Omega \Phi(x,y)q(y)   f(y) \ind y, \quad x \in \mathbb{R}^2,
\end{eqnarray*}
where $\Phi(x,y)$ is  the fundamental  solution  to the Helmholtz equation given by
\begin{equation*} \label{Green function def}
\Phi(x,y)
:=
\frac{i}{4} H^{(1)}_0(k|x-y|)
\qquad x\not=y,
\end{equation*}
here $H^{(1)}_0$ denotes the Hankel function of the first kind of order zero, cf. \cite{colton2012inverse}.
From the asymptotic behavior  of the scattered wave field (cf. \cite{cakoni2014qualitative,cakoni2016inverse})
\begin{equation} \label{full model 1}
u^{s}(x;\hat{\theta};k)
=
\frac{e^{i\frac{\pi}{4}}}{\sqrt{8k\pi}} \frac{e^{ikr}}{\sqrt{r}}\left(u^{\infty}(\hat{x};\hat{\theta};k)+\mathcal{O}\left(\frac{1}{r}\right)\right)
 \quad\mbox{as }\,r=|x|\rightarrow\infty,
\end{equation}
uniformly with respect to all directions $\hat{x}:=x/|x|\in\mathbb{S}$, we arrive at $u^{\infty}(\hat{x};\hat{\theta};k)$ which is known as the  far-field pattern with $\hat{x}\in\mathbb{S}$ denoting the observation direction. The multi-static data at a fixed frequency are given by
\begin{equation} \label{Section model far-field data}
\{u^{\infty}(\hat{x};\hat{\theta};k): \hat{x}\in\mathbb{S}, \hat{\theta}\in\mathbb{S}\}.
\end{equation}
The inverse scattering problem is to determine the contrast $q$ from these far-field data. It is known that this two dimensional inverse scattering problem has a unique solution, cf. \cite{bukhgeim2008recovering}.  We denote by $\mathcal{G}: q \mapsto u^\infty(\cdot;\cdot;k)$ the far-field map.

Motivated by recent data-driven machine learning methods for inverse scattering \cite{chen2026data,desai2025neural,khoo2019switchnet,zhou2025recovery},  we are interested in exploring the intrinsic property of the scattering data and suitable low-dimensional features. The reciprocity relation, an intrinsic property of the scattering problem, has been used to prove uniqueness of the inverse scattering problem, cf. \cite{colton2012inverse}; however it was not given enough explicit attention in the inverse algorithms. In the next section, we   process the far field data to reformulate the inverse scattering problem.

\subsection{Reciprocity-relation-aware formulation} \label{section: model reformulation}
To process the far field data, we first draw insights from the Born or linearized model.
The Born approximation $u_b^s(x;\hat{\theta};k)$ is the unique radiating solution to the Born model
\begin{equation}\label{Born model}
    \Delta u^s_b+k^2u^{s}_b=-k^2q e^{ik x \cdot \hat{\theta}} \quad \mbox{in} \quad \mathbb{R}^2.
\end{equation}
We refer  to  the model \eqref{Born model} as the Born model and the model \eqref{medium us eqn1+2} as the full model. From the asymptotic behavior
\begin{equation*}
    u_b^s(x;\hat{\theta};k)=\frac{e^{i\frac{\pi}{4}}}{\sqrt{8k\pi}}\frac{e^{ikr}}{\sqrt{r}}\left(u^{\infty}_b(\hat{x};\hat{\theta};k)+\mathcal{O}\left(\frac{1}{r}\right) \right) \quad \mbox{as}\quad r\rightarrow\infty,
\end{equation*}
we obtain  the Born far-field pattern $u^{\infty}_b(\hat{x};\hat{\theta};k)$, $\hat{x}=x/|x|\in \mathbb{S}$. One advantage of the Born far-field data is that one can directly obtain an explicit formula by
\begin{align}\label{Born fourier far field}  
    u_b^\infty (\hat{x};\hat{\theta};k)&=k^2\int_\Omega e^{-ik\hat{x}\cdot y}q(y) e^{iky\cdot \hat{\theta}}\ind y.
\end{align}
The value of the Born far-field pattern $u_b^\infty (\hat{x};\hat{\theta};k)$ is only determined by $\hat{\theta} - \hat{x}$, this motivates us to introduce   $p= \frac{\hat{\theta} - \hat{x}}{2} \in B$ where $B := B(0,1)$ denotes the unit disk, then the knowledge of the multi-static Born far-field data,  {i.e., $\{u_b^{\infty}(\hat{x};\hat{\theta};k) : \hat{x} \in \mathbb{S}, \hat{\theta} \in \mathbb{S}\}$,}
gives the knowledge of the restricted Fourier transform of the unknown $q$, i.e.,
$
\int_\Omega e^{ic py }q(y)  \ind y
$ for $p \in B$ {with $c=2k$}.

The fact that the Born far-field pattern is only determined by $\hat{\theta} - \hat{x}$ can be carried over to the fully nonlinear case, cf. \cite{cakoni2025recovery, zhou2024exploring}. We now summarize   this fact and give a brief proof. To begin with, we first state the following reciprocity relation, cf. \cite{cakoni2014qualitative}.
\begin{lemma} \label{lemma: reciprocity}
Let $u^s(x;\hat{\theta};k)$ be the unique radiating solution to \eqref{medium us eqn1+2} with $f(x)=e^{ikx\cdot \hat{\theta}}$, and let  $u^\infty(\hat{x};\hat{\theta};k)$ be its far-field pattern. Then the following reciprocity relation holds
$$
u^\infty(\hat{x};\hat{\theta};k) = u^\infty(-\hat{\theta};-\hat{x};k), \quad \forall \hat{x},\hat{\theta} \in \mathbb{S}.
$$
\end{lemma}
This motivates us to  introduce the reflection $\mathcal{R}: \mathbb{S}\times\mathbb{S} \to \mathbb{S}\times\mathbb{S}$, $\mathcal{R}(\hat{x},\hat{\theta}) := (-\hat{\theta},-\hat{x})$, and the subspace of reciprocity-symmetric functions
$$
V := \{ g \in L^2(\mathbb{S}\times\mathbb{S}): \ g \circ \mathcal{R} = g \mbox{ a.e.}\}.
$$
It is directly seen that $V$ is a closed linear subspace and Lemma \ref{lemma: reciprocity} implies that $\mathcal{G}(q) \in V$. We are now ready to prove the following property, which also introduces the linear operator mapping the far-field pattern to the processed data.
\begin{prop} \label{prop: define u(p)}
Let $c = 2k$.
\begin{itemize}
\item[]
(a)~For every $p \in B\backslash\{0\}$ there exist exactly two incident-observation pairs $(\hat{x}_j,\hat{\theta}_j)$, $j=1,2$, such that $p = \frac{\hat{\theta}_j - \hat{x}_j}{2}$, and they are related by $(\hat{x}_2,\hat{\theta}_2) = \mathcal{R}(\hat{x}_1,\hat{\theta}_1)$, i.e., $\hat{\theta}_2 = - \hat{x}_1$ and $\hat{x}_2 = - \hat{\theta}_1$. Consequently, the processed datum
$$
u(p;c) := \frac{1}{k^2} u^\infty(\hat{x};\hat{\theta};k), \qquad \mbox{where } p = \frac{\hat{\theta} - \hat{x}}{2} \mbox{ for either pair } (\hat{x},\hat{\theta}),
$$
is well defined for every $p \in B\backslash\{0\}$, and the processed data set $\{u(p;c): p \in B\}$ is uniquely defined almost everywhere.
\item[]
(b)~There exists a bounded linear operator $\mathcal{U}: L^2(\mathbb{S}\times\mathbb{S}) \to L^2(B)$ with $\|\mathcal{U}\|_{\mathcal{L}(L^2(\mathbb{S}\times\mathbb{S}),L^2(B))} \le 1/(2\sqrt{2}\,k^2)$ and $u(\cdot;c) = \mathcal{U}\mathcal{G}(q)$ a.e. in $B$.
\item[]
(c)~$\mathcal{U}$ is injective on $V$: if $g \in V$ and $\mathcal{U}g = 0$ a.e. in $B$, then $g = 0$ a.e. in $\mathbb{S}\times\mathbb{S}$.
\end{itemize}
\end{prop}
\begin{proof}
\emph{(a).} We identify $\mathbb{S}\times\mathbb{S}$ with the torus $\mathbb{T}^2 = [0,2\pi)^2$ via $\hat{x} = (\cos\varphi_1,\sin\varphi_1)$, $\hat{\theta} = (\cos\varphi_2,\sin\varphi_2)$, and let $\Lambda(\varphi_1,\varphi_2) := (\hat{\theta}-\hat{x})/2 \in \overline{B}$. A direct computation gives that the Jacobian $D\Lambda$ satisfies $\det D\Lambda = \frac{1}{4}\sin(\varphi_1-\varphi_2)$ and
  $|\det D\Lambda| = \frac{1}{4}|\sin(\varphi_2-\varphi_1)| \le \frac{1}{4}$. Let $T_\pm := \{(\varphi_1,\varphi_2) \in \mathbb{T}^2: \pm \sin(\varphi_2-\varphi_1) > 0\}$; the complement of $T_+ \cup T_-$ in $\mathbb{T}^2$ is the null set $\{\sin(\varphi_2-\varphi_1)=0\}$. In the angular variables the reflection $\mathcal{R}$ reads $(\varphi_1,\varphi_2) \mapsto (\varphi_2+\pi,\varphi_1+\pi)$ which preserves the measure and changes the sign of $\sin(\varphi_2-\varphi_1)$; this yields that $\mathcal{R}(T_\pm) = T_\mp$. 
  One can obtain a more explicit expression for $T_\pm$; indeed, for $p \in B\backslash\{0\}$,  let $\rho := |p|$, $\nu(p) := \sqrt{1-\rho^2}\,(p_2,-p_1)/\rho$, and
$$
\Gamma_+(p) := (\hat{x},\hat{\theta}) = \big(-p+\nu(p),\ p+\nu(p)\big), \quad \Gamma_-(p) := \big(-p-\nu(p),\ p-\nu(p)\big) = \mathcal{R}(\Gamma_+(p)),
$$
it follows that $\nu(p) \perp p$ and $|\nu(p)|^2 = 1-\rho^2$, so that $|\pm p + \nu(p)| = 1$ and $(\hat{\theta}-\hat{x})/2 = p$; conversely, if $(\hat{x},\hat{\theta}) \in \mathbb{S}\times\mathbb{S}$ satisfies $(\hat{\theta}-\hat{x})/2 = p$, then writing $\hat{\theta} = p+\nu$ and $\hat{x} = -p+\nu$, the conditions $|\hat{\theta}| = |\hat{x}| = 1$ force $\nu \cdot p = 0$ and $|\nu|^2 = 1-\rho^2$, i.e., $\nu = \pm \nu(p)$. This shows that the preimages of $p$ under $\Lambda$ are exactly $\Gamma_+(p)$ and $\Gamma_-(p)$. 
Finally suppose $\hat{x}_j,\hat{\theta}_j$, $j=1,2$ are such that $\Gamma_+(p) = (\hat{x}_1,\hat{\theta}_1)$ and $\Gamma_-(p) = (\hat{x}_2,\hat{\theta}_2) = (-\hat{\theta}_1,-\hat{x}_1)$, the reciprocity relation of Lemma \ref{lemma: reciprocity} gives
$$
u^\infty(\hat{x}_2;\hat{\theta}_2;k) = u^\infty(-\hat{\theta}_1;-\hat{x}_1;k) = u^\infty(\hat{x}_1;\hat{\theta}_1;k),
$$
so that $u(p;c)$ does not depend on the choice of the pair. This completes the proof.

\emph{(b).}  For $g \in L^2(\mathbb{S}\times\mathbb{S})$, we define  $\mathcal{U}: L^2(\mathbb{S}\times\mathbb{S}) \to L^2(B)$ by
$$
(\mathcal{U}g)(p) := \frac{1}{2k^2}\Big[ g(\Gamma_+(p)) + g(\Gamma_-(p)) \Big].
$$ A change of variables $p = \Lambda(\varphi)$ on $T_\pm$ yields that
$$
\int_B |g(\Gamma_\pm(p))|^2 \ind p = \int_{T_\pm} |g(\varphi)|^2 \, |\det D\Lambda(\varphi)| \ind \varphi \le \frac{1}{4}\int_{T_\pm} |g(\varphi)|^2 \ind\varphi,
$$
and consequently
$$
\|\mathcal{U}g\|_{L^2(B)} \le \frac{1}{2k^2}\cdot\frac{1}{2}\Big( \|g\|_{L^2(T_+)} + \|g\|_{L^2(T_-)} \Big) \le \frac{1}{2\sqrt{2}\,k^2}\|g\|_{L^2(\mathbb{S}\times\mathbb{S})}.
$$
For $g = \mathcal{G}(q) \in V$ (due to Lemma \ref{lemma: reciprocity}), the definition of $u(\cdot;c)$ in (a) yields that $u(\cdot;c) = \mathcal{U}\mathcal{G}(q)$ a.e. in $B$.

\emph{(c).} Let $g \in V$ with $\mathcal{U}g = 0$ a.e. in $B$. Since $g(\Gamma_{+}(p))=g(\Gamma_{-}(p))$   a.e. in $B$,   $g(\Gamma_{\pm}(p)) = 0$ a.e. in $B\backslash\{0\}$; note that the preimages of $p$ under $\Lambda$ are exactly $\Gamma_+(p)$ and $\Gamma_-(p)$, we get $g = 0$ a.e. in $T_{\pm}$.  As $\mathbb{T}^2\backslash(T_+\cup T_-)$ is a null set, $g = 0$ a.e. in $\mathbb{S}\times\mathbb{S}$. This completes the proof.
\end{proof}

The above result motivates a reformulation of the inverse scattering problem as follows. Recall $c=2k$, define the forward map
$$
\mathcal{F}: L^2(B) \to L^2(B) \qquad \mbox{by} \qquad q \longmapsto u(\cdot; c)
$$
where the processed datum $u(p;c)$ for $p\in B$ is uniquely defined almost everywhere, cf. Proposition \ref{prop: define u(p)}; in the notation of Proposition \ref{prop: define u(p)}(b), $\mathcal{F} = \mathcal{U}\mathcal{G}$. The reformulated inverse scattering problem is to determine the unknown $q \in L^2(B)$  from the processed data
\begin{equation*}\label{eqn processed data}
\{u(p): p \in B\} \mbox{ or a perturbed set } \{u^\delta(p): p \in B\}.
\end{equation*}
Correspondingly the Born processed data $\{u_b(p): p \in B\}$ are given by
\begin{align}\label{Born fourier}
    u_b(p;c)& = \int_B e^{ic p \cdot y}q(y)  \ind y.
\end{align}
It is  worth noting that the study on the  inverse problem for the Born model is important since it has a close relationship to the fully nonlinear case, cf. \cite{kirsch2017remarks} and \cite{moskow2008convergence}.
\section{A linearized low-rank structure} \label{section: linearized low-rank structure}
In the Born or linearized region, the Born data are related to the unknown $q$ via
$$
\mathcal{F}_b: L^2(B) \to L^2(B) \qquad \mbox{by} \qquad q \longmapsto u_b(\cdot; c)
$$
where $u_b(\cdot; c)$ is given by \eqref{Born fourier} and $B=B(0,1)$ denotes the unit disk.

To study this Born forward map $\mathcal{F}_b$ in the context of inverse scattering, a set of basis functions was proposed in \cite{meng23data} based on the generalizations of prolate spheroidal wave functions (PSWFs). The PSWFs and their generalizations were studied in a series of works  \cite{Slepian64,slepian1961prolate} in the 1960s. We refer to \cite{osipov2013prolate} for a comprehensive introduction to the one dimensional PSWFs and to \cite{greengard2024generalized,ZLWZ20} for more recent studies on multidimensional generalizations of the PSWFs. For the two dimensional inverse scattering problem, we rely on the generalization of  PSWFs to two dimensions \cite{Slepian64}. It is known \cite{Slepian64} that there exist real-valued  eigenfunctions $\{\psi_{m,n,l}(x;c)\}^{l\in\mathbb{I}(m)}_{m,n\in \mathbb{N}}$ of the restricted Fourier transform with parameter \rev{$c$} such that
    \begin{align}\label{eigen_R_Fourier}
        \mathcal{F}_{b}  \psi_{m,n,l}(x;c)=\int_{B}e^{ic x\cdot y}\psi_{m,n,l}(y;c) \ind y =\alpha_{m,n}(c)\psi_{m,n,l}(x;c),\quad x\in B,
    \end{align}
    where $\mathbb{N}=\{0,1,2,3,\dots\}$ and
    \begin{eqnarray*}
        \mathbb{I}(m)=\left\{
            \begin{array}{cc}
                \{1\} & m=0 \\
                \{1,2\} & m \geq 1
            \end{array}\right..
    \end{eqnarray*}
In this work we refer to $\psi_{m,n,l}(x;c)$   as the disk PSWFs and to $\alpha_{m,n}(c)$  as the prolate eigenvalues.

One of the most important properties of the disk PSWFs is  the so-called dual property. A direct calculation using \cite{Slepian64} (cf. \cite{ZLWZ20}) shows that the disk PSWFs are also eigenfunctions of a Sturm-Liouville operator, i.e.,
\begin{equation}\label{sturm-liouvill}
    \mathcal{D}  [\psi_{m,n,l}](x)=\chi_{m,n} \psi_{m,n,l}(x),\quad x\in B,
\end{equation}
where
\begin{align*}
    \mathcal{D}  := -(1-r^2)\partial_r^2-\frac{1}{r}\partial_r+3r\partial_r-\frac{1}{r^2}\Delta_0+c^2r^2
\end{align*}
and the Laplace–Beltrami operator $\Delta_0= \partial^2_\theta $ is the spherical part of Laplacian $\Delta$. More details can be found in \cite{meng23data} and \cite{ZLWZ20}. We further refer to $\chi_{m,n}(c)$  as the Sturm-Liouville eigenvalue.
The following properties of disk PSWFs can be found in \cite{Slepian64} and \cite{ZLWZ20}.
\begin{lemma} \label{lemma: SL}
Let $c>0$ be a positive real number.
\begin{itemize}
\item[]
(a)~$\{\psi_{m,n,l}(x;c)\}^{l\in\mathbb{I}(m)}_{m,n\in \mathbb{N}}$ forms a complete and orthonormal system of $L^2(B)$, i.e., for any $m,~n,~m',~n'\in\mathbb{N},~l\in\mathbb{I}(m)$ and $l'\in\mathbb{I}(m')$, it holds that
   \begin{equation*}
       \int_{B}\psi_{m,n,l}(y;c)\psi_{m',n',l'}(y;c)\ind y=\delta_{m m'}\delta_{n n'}\delta_{l l'},
   \end{equation*}
where $\delta$ denotes the Kronecker delta.
\item[]
(b)~The corresponding  Sturm-Liouville eigenvalues $\{\chi_{m,n}\}_{m,n\in \mathbb{N}}$ in \eqref{sturm-liouvill} are real positive which  are ordered for fixed $m$ as follows
       \begin{equation*}
       0<\chi_{m,0}(c)<\chi_{m,1}(c)<\chi_{m,2}(c)<\cdots .
   \end{equation*}
\item[]
(c)~Every prolate eigenvalue $\alpha_{m,n}(c)$ is  non-zero, and $\lambda_{m,n} = |\alpha_{m,n}(c)|$ can be arranged for fixed $m$ as
    \begin{equation*}
       \lambda_{m,n_1}(c)>\lambda_{m,n_2}(c)>0, \quad \forall n_1<n_2.
   \end{equation*}
   Moreover $\lambda_{m,n}(c)\longrightarrow 0$  as  $m,n\longrightarrow +\infty$.
\end{itemize}
\end{lemma}
 We emphasize that a direct evaluation of the disk PSWFs solely based on the restricted Fourier transform is not reliable as the leading prolate eigenvalues have numerically the same amplitude, cf. \cite{greengard2024generalized,ZLWZ20,zhou2024exploring}. Instead, Lemma \ref{lemma: SL} implies that  the disk PSWFs can be computed using the Sturm-Liouville differential operator to ensure stability and efficiency. Since the prolate eigenvalues decay to zero very fast, one can choose a finite dimensional set of the disk PSWFs to form a low-rank structure. We refer to \cite{meng23data} for a theoretical analysis  using such   basis functions for solving the inverse scattering problem and to \cite{zhou2024exploring} for a detailed computational treatment.

\subsection{Preliminaries about disk PSWFs}
In this section, we provide the preliminaries for the analytical and computational treatment of the  disk PSWFs. Using polar coordinates, each disk PSWF  $\psi_{m,n,l} (x;c)$ can be obtained by separation of variables (cf. \cite{greengard2024generalized,ZLWZ20})
\begin{equation*}
    \psi_{m,n,l}(x;c)={r^m}\varphi_{m,n}(2{ r}^2-1;c)Y_{m,l}(\hat{x}),\quad x\in B,
\end{equation*}
where $x=r\hat{x}=(r\cos{\theta},r\sin{\theta})^T$ and the spherical harmonics $Y_{m,l}(\hat{x})$ are given by
\begin{align}\label{spherical_harmonic}
    Y_{m,l}(\hat{x})=\left\{
            \begin{array}{cc}
                \frac{1}{\sqrt{2\pi}}, & m=0,l=1 \\
                \frac{1}{\sqrt{\pi}}\cos{m\theta}, & m\geq 1,l=1 \\
                \frac{1}{\sqrt{\pi}}\sin{m\theta},& m\geq 1,l=2
            \end{array}\right. .
\end{align}
An efficient method to evaluate the disk PSWFs is to expand $\varphi_{m,n}(\eta;c)$ by normalized Jacobi polynomials $\{P^{(m)}_{j}(\eta)\}^{j\in\mathbb{N}}_{\eta\in (-1,1)}$, i.e. $\varphi_{m,n}(\eta;c)=\sum_{j=0}^{\infty}\beta_j^{m,n}(c) P^{(m)}_j(\eta)$.
With the help of the Sturm-Liouville problem \eqref{sturm-liouvill}, the  coefficients $\{\beta_j^{m,n}(c)\}$ can be solved via a tridiagonal linear system, cf.  \cite{greengard2024generalized,ZLWZ20} and \cite{zhou2024exploring}. Here the normalized Jacobi polynomials $\{P_n^{(m)}(x)\}_{x\in (-1,1) }$ can be obtained through the three-term recurrence relation
\begin{align*}
         P^{(m)}_{n+1}(x)&=\frac{1}{a_n} [ (x-b_n) P_n^{(m)}(x)-a_{n -1}P^{(m)}_{n-1}(x) ],\quad n\geq 1\\
    P^{(m)}_{0}(x)&=\frac{1}{h_0 },\quad P^{(m)}_{1}(x)=\frac{1}{2h_1  }[(m+2)x-m],
\end{align*}
where $h_{0}= \frac{1}{\sqrt{2(m+1)}}$, $h_{1}= \frac{1}{\sqrt{2(m+3)}}$, and
\begin{align*}
   \left\{
            \begin{array}{cc}
                a_{n}=&\frac{2(n+1)(n+m+1)}{(2 n+m+2)\sqrt{(2 n+{m+1})(2 n+m+3)}} \\
                b_{n}=&\frac{m^{2}}{(2 n+m)(2 n+m+2)}
            \end{array}\right. , \qquad n\in\mathbb{N} \mbox{ for } a_n \mbox{ and } n \ge 1 \mbox{ for } b_n.
\end{align*}
For a more comprehensive introduction to special polynomials, we refer to \cite{Abramowitz64}.

 \section{Stability estimate in a low-rank space} \label{section: stability}
 In this section, we investigate the Lipschitz stability in a low-rank space spanned by finitely many disk PSWFs.
The specific choice of the disk PSWFs becomes essential in Remark \ref{remark: approximation} (approximation rates via the dual property) and in Section \ref{section: explicit constant} (explicit constant in the linearized region).
We first recall the following abstract Lipschitz stability for inverse problems given in \cite{bourgeois2013remark} (see also \cite{alberti2022infinite}).
\begin{lemma}\label{lemma abstract Lipschitz stability inverse}
    Let $X$ and $Y$ be Banach spaces. Let $A \subseteq X$ be an open subset, $W \subseteq X$ be a finite-dimensional subspace and $K \subseteq W \cap A$ be a compact and convex subset. Let the operator $\mathcal{Y} \in C^1(A,Y)$ be such that $\mathcal{Y}|_{W \cap A}$ and $\mathcal{Y}'(x)|_W$, $x \in W \cap A$, are injective, where $\mathcal{Y}'$ denotes the Fr\'{e}chet derivative.

    Then there exists a constant $C>0$ such that
    $$
    \|x_1 - x_2\|_X \le C \|\mathcal{Y} (x_1) - \mathcal{Y}(x_2)\|_Y, \qquad x_1,x_2 \in K.
    $$
\end{lemma}
In general, the Lipschitz constant $C$ tends to infinity as $\dim(W) \to +\infty$ for ill-posed problems. Lemma \ref{lemma abstract Lipschitz stability inverse} can be   applied to the inverse medium scattering problem with far-field data, cf. \cite{bourgeois2013remark}. Specifically, let $X=L^\infty (B)$, $Y=L^2(\mathbb{S} \times \mathbb{S})$, $A = L^\infty_+(B):=\{ q \in L^\infty(B): \mbox{Im}(q) \ge \lambda \mbox{ in } B \mbox{ for some } \lambda>0\}$, $W$ be a finite-dimensional subspace of $L^\infty(B)$ and $K$ be a convex and compact subset of $W \cap A$. If  the following Assumption \ref{assumption: low-rank stability G} holds, then it follows  that
   \begin{equation} \label{eqn Lipschitz stability inverse medium}
        \|q_1 - q_2\|_{L^\infty(B)} \le C \|u^\infty_{1} - u^\infty_{2}\|_{L^2(\mathbb{S} \times \mathbb{S})} = C \|\mathcal{G}(q_1) - \mathcal{G}(q_2)\|_{L^2(\mathbb{S}\times\mathbb{S})}, \quad q_1,q_2 \in K,
   \end{equation}
    where $C>0$ is a constant,    $u^\infty_{j}$ denotes the far-field pattern for $q_j$, $j=1,2$, and $\mathcal{G}$ denotes the far-field map.

\begin{assumption}\label{assumption: low-rank stability G}
    Assume that $\mathcal{G} \in C^1(L^\infty_+(B); L^2(\mathbb{S}\times\mathbb{S}))$,   $\mathcal{G}|_{W\cap L^\infty_+(B)}$ is injective, and  $\mathcal{G}'(q)|_W$ is injective for every $q \in W \cap L^\infty_+(B)$. 
\end{assumption}
Indeed, Assumption \ref{assumption: low-rank stability G} is verified in \cite{bourgeois2013remark} in three dimensions. In two dimensions it may be verified as follows. For every $q \in L^\infty_+(B)$ the direct problem is uniquely solvable, thereby the far-field map $\mathcal{G}$ of Section \ref{section: model reformulation} is well defined on $L^\infty_+(B)$, and its $C^1$ regularity follows from the smooth dependence on $q$ of the solution of the Lippmann-Schwinger equation (cf. \cite{colton2012inverse,kirsch2017remarks}). The injectivity of $\mathcal{G}$ on $L^\infty_+(B)$ is the uniqueness of the two-dimensional inverse medium problem for $L^p$ potentials with $p>2$, cf. \cite{bukhgeim2008recovering,blasten2015stability}. Finally, $(\mathcal{G}'(q)h)(\hat{x};\hat{\theta}) = k^2\int_B h(y)\,u_q^t(y;\hat{\theta})\,u_q^t(y;-\hat{x})\,\ind y$, where $u^t_q(\cdot;\hat{\theta})$ denotes the total field corresponding to $q$, so that the injectivity of $\mathcal{G}'(q)|_W$ follows from the completeness of the total fields $\{u^t_q(\cdot;\hat{\theta}):\hat{\theta}\in\mathbb{S}\}$ in the $L^2(B)$-closure of the solutions to $\Delta v + k^2(1+q)v=0$ in $B$ (cf. \cite{colton2012inverse}) together with the completeness of products of such solutions (cf. \cite{blasten2015stability}). Hereon, we impose Assumption \ref{assumption: low-rank stability G} in the remainder of the paper.

Due to the fact that the disk PSWFs are the eigenfunctions of the Born forward operator \eqref{Born fourier}, it is natural to look for unknowns in a low-rank space spanned by finitely many disk PSWFs. We now prove the following theorem.
    \begin{theorem} \label{theorem Lipschitz in low-rank uknown} Given a positive constant $\eta>0$,
        let $W = \mbox{span}\{ \psi_{m,n,\ell}(\cdot;c): {|\alpha_{m,n}(c)|} > \eta\}$ and $K$ be a convex and compact subset of $W \cap L^\infty_+(B)$. Suppose that Assumption \ref{assumption: low-rank stability G} holds. Let $u_j(\cdot;c)$ be defined via Proposition \ref{prop: define u(p)} for the contrast $q_j$, $j=1,2$. Then there exists a constant $C_0>0$ such that
        $$
    \|q_1 - q_2\|_{L^\infty(B)} \le C_0 \|u_1(\cdot;c) - u_2(\cdot;c)\|_{L^2(B)}, \qquad q_1,q_2 \in K,
    $$
    and consequently $\|q_1 - q_2\|_{L^2(B)} \le \sqrt{\pi}\, C_0 \|u_1(\cdot;c) - u_2(\cdot;c)\|_{L^2(B)}$ for all $q_1,q_2 \in K$.
    \end{theorem}
    \begin{proof}
    Let $\mathcal{U}$ be the operator of Proposition \ref{prop: define u(p)}. By Proposition \ref{prop: define u(p)}(b), the forward map of Section \ref{section: model reformulation} factors as $\mathcal{F} = \mathcal{U}\mathcal{G}$ on $A$, and $\mathcal{F}(q_j) = u_j(\cdot;c)$ a.e. in $B$. We  apply Lemma \ref{lemma abstract Lipschitz stability inverse} with $X = L^\infty(B)$, $Y = L^2(B)$, $\mathcal{Y} = \mathcal{F}$ and the given $A$, $W$ and $K$; it remains to verify the assumptions. Since the   quotient  $s^{-1}(\mathcal{G}(q+sh) - \mathcal{G}(q))$ belong to $V$ for sufficiently small $s>0$, and the quotient  converge   to $\mathcal{G}'(q)h$ as $s \to 0$, then $\mathcal{G}'(q)h \in V$. 

    (i) $\mathcal{F} \in C^1(A; L^2(B))$. Since $\mathcal{U}$ is bounded and linear, the chain rule shows that $\mathcal{F}$ is Fr\'echet differentiable with $\mathcal{F}'(q) = \mathcal{U}\mathcal{G}'(q)$, and
    $$
    \|\mathcal{F}'(q) - \mathcal{F}'(\tilde{q})\|_{\mathcal{L}(L^\infty(B), L^2(B))} \le \|\mathcal{U}\|_{\mathcal{L}(L^2(\mathbb{S}\times\mathbb{S}),L^2(B))}\, \|\mathcal{G}'(q) - \mathcal{G}'(\tilde{q})\|_{\mathcal{L}(L^\infty(B), L^2(\mathbb{S}\times\mathbb{S}))},
    $$
    so that $q \mapsto \mathcal{F}'(q)$ is continuous because $\mathcal{G} \in C^1(A; L^2(\mathbb{S}\times\mathbb{S}))$.

    (ii) $\mathcal{F}|_{W\cap A}$ is injective. Let $q_1,q_2 \in W \cap A$ with $\mathcal{F}(q_1) = \mathcal{F}(q_2)$. Then $\mathcal{U}(\mathcal{G}(q_1) - \mathcal{G}(q_2)) = 0$ a.e. in $B$, and $\mathcal{G}(q_1) - \mathcal{G}(q_2) \in V$ by Lemma \ref{lemma: reciprocity}. Proposition \ref{prop: define u(p)}(c) yields $\mathcal{G}(q_1) = \mathcal{G}(q_2)$, and the injectivity of $\mathcal{G}|_{W\cap A}$ yields $q_1 = q_2$.

    (iii) $\mathcal{F}'(q)|_W$ is injective for every $q \in W \cap A$. Let $h \in W$ with $\mathcal{F}'(q)h = \mathcal{U}\mathcal{G}'(q)h = 0$.  Proposition \ref{prop: define u(p)}(c) and $\mathcal{U}\mathcal{G}'(q)h = 0$ yield  that $\mathcal{G}'(q)h = 0$, and the injectivity of $\mathcal{G}'(q)|_W$ yields  $h = 0$.

    Now we can apply Lemma \ref{lemma abstract Lipschitz stability inverse} to conclude that there exists a constant $C_0>0$ such that
    $$
    \|q_1 - q_2\|_{L^\infty(B)} \le C_0 \|\mathcal{F}(q_1) - \mathcal{F}(q_2)\|_{L^2(B)}, \qquad q_1,q_2 \in K,
    $$
    which is the first estimate because $\mathcal{F}(q_j) = u_j(\cdot;c)$. The second estimate follows from $\|\cdot\|_{L^2(B)} \le |B|^{1/2}\|\cdot\|_{L^\infty(B)}$ with $|B| = \pi$. This completes the proof.
    \end{proof}
    
Furthermore, the processed data $u(\cdot;c)$ can also be approximated in a low-rank space $\mbox{span}\{ \psi_{m,n,\ell}(\cdot;c): {|\alpha_{m,n}(c)|} > \zeta\}$. We now prove the Lipschitz stability where both the unknown and data belong to low-rank spaces. For the general Lipschitz stability with finite measurements in low-rank spaces, we refer to \cite{alberti2022infinite}.
 \begin{theorem}\label{theorem Lipschitz in low-rank space unknown and data}
Under the assumptions in Theorem \ref{theorem Lipschitz in low-rank uknown}, for $\zeta>0$ let $\mathcal{P}_\zeta$ be the orthogonal projection of $L^2(B)$ onto $\mbox{span}\{ \psi_{m,n,\ell}(\cdot;c): {|\alpha_{m,n}(c)|} \ge \zeta\}$, i.e.,
$$
\mathcal{P}_\zeta v =\sum_{|\alpha_{m,n}(c)| \ge \zeta}   \left\langle v,  \psi_{m,n,\ell}(\cdot;c) \right\rangle_{B}  \psi_{m,n,\ell}(\cdot;c), \qquad v \in L^2(B),
$$
where $\langle \cdot, \cdot \rangle_B$ denotes the inner product of $L^2(B)$. Then there exist  constants $C_1>0$ and $\tilde{\zeta}>0$ such that for all $0 < \zeta \le \tilde{\zeta} $,
        $$
    \|q_1 - q_2\|_{L^\infty(B)} \le C_1  \|\mathcal{P}_\zeta (u_1(\cdot;c)) - \mathcal{P}_\zeta (u_2(\cdot;c))\|_{L^2(B)}, \qquad q_1,q_2 \in K.
    $$
\end{theorem}
\begin{proof}
We first apply \cite[Corollary 1]{alberti2022infinite} with $X = L^\infty(B)$, $Y = L^2(B)$, $A = L^\infty_+(B)$, the finite-dimensional space $W$, the compact convex set $K$, the map $\mathcal{F} = \mathcal{U}\mathcal{G} \in C^1(A; L^2(B))$, the injectivity of $\mathcal{F}|_{W \cap A}$ and $\mathcal{F}'(q)|_{W}$, $Q_N := \mathcal{P}_{\zeta_N}$ for any sequence $\zeta_N \downarrow 0$. To apply \cite[Corollary 1]{alberti2022infinite}, we check the  Hypothesis (a) of \cite[Corollary 1]{alberti2022infinite}: the fact that $\mathcal{P}_\zeta \to \mathcal{I}$ strongly as $\zeta \to 0$ follows because $\alpha_{m,n}(c)\not=0$, $|\alpha_{m,n}(c)| \to 0$  and the disk PSWFs are complete (cf. Lemma \ref{lemma: SL}). Then \cite[Corollary 1]{alberti2022infinite}  yields that, there exists $N>0$ and a constant $C_1>0$ such that        $
    \|q_1 - q_2\|_{L^\infty(B)} \le C_1  \|\mathcal{P}_{\zeta_N} (u_1(\cdot;c)) - \mathcal{P}_{\zeta_N} (u_2(\cdot;c))\|_{L^2(B)}$,  for all $q_1,q_2 \in K
    $.  This shows the inequality for $\zeta=\tilde{\zeta}:=\zeta_{N}$. Finally, for $0 < \zeta \le \tilde{\zeta}$, the ranges of $\mathcal{P}_{\tilde{\zeta}}$ is contained in the range of $\mathcal{P}_\zeta$ and it follows that $\|\mathcal{P}_{\tilde{\zeta}} v\|_{L^2(B)} \le \|\mathcal{P}_\zeta v\|_{L^2(B)}$ for all $v \in L^2(B)$, which gives the estimate for all $0 < \zeta \le \tilde{\zeta}$. This completes the proof.
\end{proof}
\begin{remark} \label{remark: approximation}
The disk PSWFs are the eigenfunctions of not only the Born forward operator \eqref{Born fourier} but also the differential operator \eqref{sturm-liouvill}. This dual property allows the quantification of the approximation capability of the proposed low-rank space in inverse scattering.  More precisely, for a given $\eta \in (0, |\alpha_{0,0}|)$, let $\beta_\eta>0$ be the smallest $\beta$ such that $\{(m,n,\ell): \chi_{m,n}(c) \le \beta^{-1} \} \subseteq \{(m,n,\ell): |\alpha_{m,n}(c)| \ge  \eta\}$  which can be chosen since $\chi_{m,n}(c)$ grows to infinity and $\alpha_{m,n}(c)$ decays to zero. Let $\mathcal{P}_\eta u$ be the projection of $u$ onto the finite dimensional space $\mbox{span}\{\psi_{m,n,\ell}: |\alpha_{m,n}(c)|\ge\eta\}$
 \begin{eqnarray*}
\mathcal{P}_\eta u := \sum_{|\alpha_{m,n}(c)| \ge  \eta}  \left\langle u, \psi_{m,n,\ell}(\cdot;c) \right\rangle  \psi_{m,n,\ell}(\cdot;c),
\end{eqnarray*}
then for any $q \in H^s(B)$ and $s \in (0,1)$, it holds that
{\small
\begin{eqnarray*} \label{section SLP theorem u H^1 projection error}
\|  \mathcal{P}_\eta q - q \| \le \|  \sum_{|\chi_{m,n}(c)| \le  \beta_\eta^{-1}}  \left\langle q, \psi_{m,n,\ell}(\cdot;c) \right\rangle  \psi_{m,n,\ell}(\cdot;c) - q \|  \le  (\beta_\eta C)^{s/2}    (1+c^2)^{s/2}  \|q\|_{H^s(B)},
\end{eqnarray*}
}
for some positive constant $C\ge\sqrt{3}$ independent of $q$, $s$ and $c$; here the last inequality follows from \cite{meng23data}.
In principle, the better regularity of $q$ the better approximation capability in the proposed low-rank space.
\end{remark}

\subsection{Explicit Lipschitz constant in the linearized region} \label{section: explicit constant}
One of the advantages of the proposed low-rank space is that one can obtain the following explicit Lipschitz estimate, in terms of computable eigenvalues.
\begin{theorem} \label{lemma: Lipschitz stability in low-rankspace}
Let  $q_j \in \mbox{span}\{ \psi_{m,n,\ell}(\cdot;c): {|\alpha_{m,n}(c)|}  \ge  \eta\}$ with $\eta>0$,  and $u_{b,j}$ be the (processed) Born approximation given by \eqref{Born fourier}, $j=1,2$. Then  $q_j$ is recovered from $u_{b,j}$ by the truncated inverse Born formula  
\begin{equation} \label{eqn Lipschitz stability inverse medium Born}
  q_j = \sum_{|\alpha_{m,n}(c)| \ge \eta}   \frac{1}{ \alpha_{m,n}(c) } \left\langle {   u_{b,j}},  \psi_{m,n,\ell}(\cdot;c) \right\rangle_{B}  \psi_{m,n,\ell}(\cdot;c), \qquad j=1,2,
\end{equation}
and the following Lipschitz stability holds
$$
\|{ q_1} - q_2\|_{L^2(B)}  \le  { \frac{1}{\eta}} \|u_{b,1} - u_{b,2}\|_{L^2(B)}.
$$
\end{theorem}
\begin{proof}
Note that the disk PSWFs are the eigenfunctions of the restricted Fourier transform \eqref{eigen_R_Fourier}, then the stability estimate follows directly from
\begin{eqnarray*}
   &&\|{ q_1} - q_2\|^2_{L^2(B)}
   = \sum_{|\alpha_{m,n}(c)| \ge \eta}   \frac{1}{ |\alpha_{m,n}(c)|^2 } \Big|\left\langle {   u_{b,1}},  \psi_{m,n,\ell}(\cdot;c) \right\rangle   - \left\langle {   u_{b,2}},  \psi_{m,n,\ell}(\cdot;c) \right\rangle  \Big|^2 \\
   & \le & \sum_{|\alpha_{m,n}(c)| \ge \eta}   \frac{1}{ \eta ^2 } \Big|\left\langle {   u_{b,1}},  \psi_{m,n,\ell}(\cdot;c) \right\rangle   - \left\langle {   u_{b,2}},  \psi_{m,n,\ell}(\cdot;c) \right\rangle  \Big|^2 \le  { \frac{1}{\eta^2}} \|u_{b,1} - u_{b,2}\|^2_{L^2(B)}.
\end{eqnarray*}
This completes the proof.
\end{proof}
According to Theorem  \ref{theorem Lipschitz in low-rank space unknown and data}, one can  choose  $q \in  \mbox{span}\{ \psi_{m,n,\ell}(\cdot;c): |\alpha_{m,n}(c)| > \eta\}$ and $u_b \in  \mbox{span}\{ \psi_{m,n,\ell}(\cdot;c): {|\alpha_{m,n}(c)|} > \zeta\}$; however since the disk PSWFs  are exactly the eigenfunctions of the Born forward operator \eqref{Born fourier}, one can simply choose  $\eta=\zeta$ in Theorem \ref{lemma: Lipschitz stability in low-rankspace}.

\begin{remark}
We remark that the low-rank structure plays an important role in establishing explicit  conditional  a priori estimates in the spirit of  increasing stability or H\"older-logarithmic stability in the linear case, cf.  \cite{novikov22} using 1d low-rank structure in connection with Radon transform and \cite{meng23data} using the 2d low-rank structure.
\end{remark}
Here we emphasize the unique feature again that  the leading prolate eigenvalues have numerically the same amplitude, cf. \cite{greengard2024generalized,ZLWZ20,zhou2024exploring} and that the prolate eigenvalues decay to zero very fast.
For a given cut off $\eta$, it is seen that the dimension of the chosen low-rank space $\mbox{span}\{ \psi_{m,n,\ell}(\cdot;c): {|\alpha_{m,n}(c)|} > \eta\}$ will be determined by the wave number $k=c/2$, unlike other heuristic ways of choosing the dimension of low-rank spaces.    We will use the linearized low-rank structure to find an inverse Born approximation $q_0$ via \eqref{eqn Lipschitz stability inverse medium Born} as an initial guess. In this linearized region, there are similar works to obtain $q_0$ using the one dimensional PSWF and Radon transform \cite{novikov22,isaev2022numerical}, a modification of the linear sampling method \cite{audibert2024shape}, and a training-free kernel machine  approach \cite{mengzhang24}.

Motivated by the above Lipschitz stability estimate,
in this work we further pursue a low-rank-assisted ensemble Kalman filter to update the numerical solution iteratively. In the later sections, we conveniently drop the parameter $c$ in $\psi_{m,n,\ell}(\cdot;c)$ and $\alpha_{m,n}(c)$ for best readability when there is no confusion.

\section{Ensemble Kalman filter for the  inverse  scattering problem} \label{section: EnKF}
To discuss the  ensemble Kalman filter, we introduce the following notations. Let $X$ be a separable real Hilbert space equipped with the inner product $\langle \cdot, \cdot\rangle_X$ and the norm $\|\cdot\|_X$. A semi-positive definite operator $\mathcal{T}: X \to X$ is such that $\langle \mathcal{T} x, x \rangle_X \ge 0$ for any $x \in X$. Let $\mathcal{C}: X \to X$ be a compact, self-adjoint and semi-positive definite linear operator, we say $\mathcal{C}$ is a trace class operator if
$$
\sum_{k=1}^\infty \langle \mathcal{C} e_k, e_k \rangle_X < \infty
$$
where $\{e_k\}_{k=1}^\infty$ is any orthonormal basis of $X$. For such operators the definition of the trace class operator is independent of the specific choice of an orthonormal basis, cf. \cite{ringrose1971compact}; all the covariance operators considered below are of this type.

Let $(\Omega_0, \mathcal{A}, \mathbb{P})$ be a probability space. A mapping $\mathcal{X}: \Omega_0 \to X$ is called a random element if $\mathcal{X}$ is $(\mathcal{A}, \mathbb{B})$-measurable, where $\mathbb{B} = \sigma(\mathbb{O})$ is the $\sigma$-field of Borel sets with $\mathbb{O}$ denoting the collection of open sets in $X$. In short, we call $\mathcal{X}$  an  $X$-valued random element.
The expectation of  an  $X$-valued random element $\mathcal{Q}$ is defined by
$$
E[\mathcal{Q}] = \int_{\Omega_0} \mathcal{Q}(\omega) \ind \mathbb{P}(\omega)
$$
in the sense of Bochner integral, cf. \cite{yosida2012functional}. It can be shown that the expectation satisfies
$$
\langle E[\mathcal{Q}], x \rangle_X = E[\langle \mathcal{Q}, x \rangle_X] \quad \mbox{for any} \quad x \in X,
$$
and $E[\mathcal{Q}]$ exists if $E[\|\mathcal{Q}\|] < \infty$.

Suppose $E[\|\mathcal{Q}\|^2] < \infty$, then the covariance operator $\mbox{Cov}(\mathcal{Q}):X \to X$ is  the  unique self-adjoint, semi-positive definite, trace class operator such that
$$
\langle \mbox{Cov}(\mathcal{Q}) x_1, x_2 \rangle = E[\langle x_1,  \mathcal{Q} - E[\mathcal{Q}]  \rangle \langle \mathcal{Q} - E[\mathcal{Q}], x_2 \rangle] \quad \mbox{for any} \quad x_1,x_2 \in X.
$$
For later purposes we define, for any two elements $q_1 \in X$ and $q_2 \in Y$ where $Y$ is a real Hilbert space, the tensor product $q_1 \otimes q_2: Y \to X$  by
$$
\langle (q_1 \otimes q_2) y, x \rangle_X =  \langle y, q_2\rangle_Y \langle  q_1, x\rangle_X \quad \mbox{for any} \quad x \in X, y \in Y.
$$

The goal of the next section is to, \textit{heuristically}, introduce and interpret the ensemble Kalman filter via  the classical Tikhonov-Phillips regularization, cf. \cite{nakamura2015inverse} and \cite{parzer2022convergence}.
 \subsection{Tikhonov-Phillips regularization with low-rank approximation}
 To begin with, let us briefly recall the key ideas of Tikhonov-Phillips regularization \cite{phillips1962technique,tikhonov1963solution,tikhonov1963regularization} with adaptation to our inverse problem. Since it is easy to work with real-valued Hilbert space for the ensemble Kalman filter, we introduce $L^2(B;\mathbb{R})$ (resp. $L^2(B)=L^2(B;\mathbb{C})$)  the real-valued (resp. complex-valued) $L^2(B)$ space. Furthermore let   $\mathcal{K}_1: L^2(B;\mathbb{R}) \oplus L^2(B;\mathbb{R}) \to L^2(B;\mathbb{C})$ be given by
 \begin{equation} \label{eqn: def K1}
 \mathcal{K}_1 \left(\begin{matrix}
     q_{R} \\
     q_{I}
 \end{matrix} \right)  =  q_{R} + i q_{I}, \qquad \forall \,\left(\begin{matrix}
     q_{R} \\
     q_{I}
 \end{matrix} \right) \in L^2(B;\mathbb{R}) \oplus L^2(B;\mathbb{R}),
 \end{equation}
 where $X := L^2(B;\mathbb{R}) \oplus L^2(B;\mathbb{R})$ denotes the direct sum and hence a Hilbert space equipped with inner product
 $$
 \left\langle \left(\begin{matrix}
     q^1_{R} \\
     q^1_{I}
 \end{matrix} \right) , \left(\begin{matrix}
     q^2_{R} \\
     q^2_{I}
 \end{matrix} \right)  \right\rangle_X =  \langle q^1_R, q^2_R \rangle_{L^2(B;\mathbb{R})}  +  \langle q^1_I, q^2_I \rangle_{L^2(B;\mathbb{R})}.
 $$
 We also introduce $\mathcal{K}_2: L^2(B;\mathbb{C}) \to L^2(B;\mathbb{R}) \oplus L^2(B;\mathbb{R})$ by
 \begin{equation} \label{eqn: def K2}
 \mathcal{K}_2 (q) =  \left(\begin{matrix}
     \mbox{Re}(q) \\
     \mbox{Im}(q)
 \end{matrix} \right), \qquad \forall q \in L^2(B;\mathbb{C}).
 \end{equation}
 This convenient notation allows us to introduce an equivalent forward map as $\mathcal{K}_2\mathcal{F} \mathcal{K}_1:X \to X$ where $X = L^2(B;\mathbb{R}) \oplus L^2(B;\mathbb{R})$ and
 $$
q=\left(\begin{matrix}
     \mbox{Re}(q) \\
     \mbox{Im}(q)
 \end{matrix} \right) \mapsto \mathcal{K}_2\mathcal{F} \mathcal{K}_1 q = \left(\begin{matrix}
     \mbox{Re}(\mathcal{F}\mathcal{K}_1q) \\
     \mbox{Im}(\mathcal{F}\mathcal{K}_1q)
 \end{matrix} \right).
 $$
The definition of $\mathcal{K}_1$ and $\mathcal{K}_2$ extends similarly to $Y \oplus Y$ when $Y$ is a finite dimensional subspace of $L^2(B;\mathbb{R})$. Hereon we identify $q$ and the processed data as functions in $X$.

 In this work, we assume that we are given a trace class operator $\mathcal{C}$ that is always injective, self-adjoint and semi-positive definite. Following the notation in \cite{nakamura2015inverse}, we define $\langle x_1, x_2\rangle_{X,\mathcal{C}^{-1}}= \langle \mathcal{C}^{-1/2} x_1, \mathcal{C}^{-1/2}x_2\rangle_{X}$ for any $x_1$ and $x_2$ in $\mbox{Range}(\mathcal{C}^{1/2})$, and the corresponding norm by $\|\cdot\|_{X,\mathcal{C}^{-1}}$.
 Given an initial guess $q_0 \in X$, suppose that the forward operator $\mathcal{F}$ has the following approximation
 $$
 \mathcal{K}_2\mathcal{F} \mathcal{K}_1 q \approx \mathcal{K}_2 \mathcal{F} \mathcal{K}_1 q_0 + \mathcal{L}_0 \Delta q,
 $$
 where $\mathcal{L}_0: X \to X$ is  a linear bounded operator. To solve $\mathcal{K}_2 \mathcal{F}\mathcal{K}_1 q \approx u^\delta$ with noisy data $u^\delta \in X$,   consider the Tikhonov-Phillips regularization to solve the following  optimization problem
 $$
 \Delta \widetilde{q}^{(\infty)} = \argmin_{\Delta q \in \mbox{\footnotesize Range}(\mathcal{C}^{1/2})} \Big( \| \mathcal{L}_0 \Delta q -(u^\delta- \mathcal{K}_2 \mathcal{F} \mathcal{K}_1 q_0) \|^2_{X} + 1/N_{\rm e}\|\Delta q\|^2_{X,\mathcal{C}^{-1}} \Big)
 $$
 where  $N_{\rm e}$ is a positive integer. It is known (cf. \cite{nakamura2015inverse}) that the solution is given by
 \begin{equation}\label{eqn: TP regularization  qinfty}
  \Delta \widetilde{q}^{(\infty)} = \mathcal{C} \mathcal{L}_0^* \Big(\mathcal{L}_0 \mathcal{C} \mathcal{L}_0^* + \mathcal{I}/N_{\rm e}  \Big)^{-1} (u^\delta - \mathcal{K}_2 \mathcal{F} \mathcal{K}_1 q_0),
 \end{equation}
where $\mathcal{I}$ denotes the identity operator.

One can also work with low-rank approximations \{$\mathcal{\widetilde{C}}^{(M)}: X \to X\}_{M=1}^\infty$ of  $\mathcal{C}$ such that
$$
\lim_{M \to \infty} \|\mathcal{\widetilde{C}}^{(M)} -\mathcal{C}\|_{\mathcal{L}(X,X)} =0,
$$
where each $\mathcal{\widetilde{C}}^{(M)}$ is an injective, self-adjoint, semi-positive definite  trace class operator,
and $\mbox{Range}((\mathcal{\widetilde{C}}^{(M)})^{1/2}) \subset \mbox{Range}(\mathcal{C}^{1/2})$.
The solution to the  Tikhonov-Phillips regularization problem
{\small
 \begin{equation} \label{eqn: TP regularization Delta qM}
   \Delta \widetilde{q}^{(M)} = \argmin_{\Delta q \in \mbox{\footnotesize Range}((\mathcal{\widetilde{C}}^{(M)})^{1/2})} \Big( \| \mathcal{L}_0 \Delta q -(u^\delta- \mathcal{K}_2 \mathcal{F} \mathcal{K}_1 q_0) \|^2_{X} + 1/N_{\rm e}\big\|\Delta q\big\|^2_{X,\mathcal{\widetilde{C}}^{(M),-1}} \Big),
 \end{equation}
 }
is given by
 \begin{equation} \label{eqn: TP regularization  qM}
 \Delta \widetilde{q}^{(M)} = \mathcal{\widetilde{C}}^{(M)} \mathcal{L}_0^* \Big(\mathcal{L}_0 \mathcal{\widetilde{C}}^{(M)} \mathcal{L}_0^* +   \mathcal{I}/N_{\rm e}\Big)^{-1} \Big(u^\delta - \mathcal{K}_2 \mathcal{F} \mathcal{K}_1 q_0\Big).
\end{equation}
 Let
$$
 \widetilde{q}^{(\infty)} = q_0 + \Delta \widetilde{q}^{(\infty)} \quad \mbox{and} \quad  \widetilde{q}^{(M)} = q_0 + \Delta \widetilde{q}^{(M)}.
$$
 The following lemma, see for instance \cite{nakamura2015inverse} and \cite{parzer2022convergence}, gives the relation between $\Delta \widetilde{q}^{(\infty)}$ and $\Delta \widetilde{q}^{(M)}$.

 \begin{lemma}
 The solution $\widetilde{q}^{(M)}$ converges to $ \widetilde{q}^{(\infty)}$ as $M \to \infty$. Specifically
 $$
 \|  \widetilde{q}^{(M)} -   \widetilde{q}^{(\infty)}\|_X=\|\Delta \widetilde{q}^{(M)} - \Delta \widetilde{q}^{(\infty)}\|_X \le \zeta_0 N_{\rm e}^2 \|\mathcal{\widetilde{C}}^{(M)} - \mathcal{C}\|_{\mathcal{L}(X,X)}
 $$
 where $\zeta_0$ is a positive constant independent of $M$.
 \end{lemma}
 \begin{proof}
 For  any semi-positive and self-adjoint  linear operators $\mathcal{T}_1 \in \mathcal{L}(X,X)$ and $\mathcal{T}_2 \in \mathcal{L}(X,X)$, it holds   that
 $$
 \|(\mathcal{T}_1 +   \mathcal{I}/N_{\rm e})^{-1}\|_{\mathcal{L}(X,X)} \le N_{\rm e} \mbox{ and }
 $$
  $$
  \|(\mathcal{T}_1 +  \mathcal{I}/N_{\rm e})^{-1} - (\mathcal{T}_2 +   \mathcal{I}/N_{\rm e})^{-1}\|_{\mathcal{L}(X,X)} \le N_{\rm e}^2 \| \mathcal{T}_1- \mathcal{T}_2\|_{\mathcal{L}(X,X)}.
 $$
Then subtracting \eqref{eqn: TP regularization  qinfty} from \eqref{eqn: TP regularization  qM} and using the above two equations, one can directly prove the result. This completes the proof.
 \end{proof}

 \subsection{Ensemble Kalman filter and connection to Tikhonov-Phillips regularization} \label{section: linear general EnKF}
 Now we are ready to introduce
 the  ensemble Kalman filter using its connection to the Tikhonov-Phillips regularization, cf. three-dimensional and four-dimensional variational data assimilation (i.e., 3D-VAR and 4D-VAR) \cite{nakamura2015inverse}. To begin with, given  an  $X$-valued ensemble
 $$
 \Delta Q^{(M),0}=[\Delta Q^{(M),0}_1, \Delta Q^{(M),0}_2, \cdots, \Delta Q^{(M),0}_M]
 $$
 where each $X$-valued random element $\Delta Q^{(M),0}_m$, $m=1,2,\cdots, M$, is chosen according to mean $0$ and covariance $\mathcal{C}$.
Then the ensemble Kalman filter iterations are as follows.

\noindent\textbf{Iteration ($n\to n+1$ until $N_{\rm e}$):} Compute
$$
\Delta W_m^{(M),0} = \mathcal{L}_0 \Delta Q_m^{(M),0}, \quad m=1,2,\cdots,M,
$$
and the  empirical  ensemble mean
$$
{ E_M}\big[ \Delta Q^{(M),0} \big] = \frac{1}{M} \sum_{m=1}^M \Delta Q_m^{(M),0}, \quad { E_M}\big[ \Delta W^{(M),0} \big] = \frac{1}{M} \sum_{m=1}^M \Delta W_m^{(M),0}.
$$
\rev{Here and in the following, $E_M[\cdot]$ denotes the empirical mean over the ensemble, which is to be distinguished from the expectation $E[\cdot]$.}
Generate two operators $\mathcal{T}_{qw}^{(M),0}: X \to X$ and $\mathcal{T}_{ww}^{(M),0}: X \to X$ via
\begin{eqnarray*}
    \mathcal{T}_{qw}^{(M),0} &=& \frac{1}{M} \sum_{m=1}^M (\Delta Q_m^{(M),0} - { E_M}\big[ \Delta Q^{(M),0} \big]) \otimes (\Delta W_m^{(M),0} - { E_M}\big[ \Delta W^{(M),0} \big]),  \\
       \mathcal{T}_{ww}^{(M),0} &=& \frac{1}{M} \sum_{m=1}^M (\Delta W_m^{(M),0} - { E_M}\big[ \Delta W^{(M),0} \big]) \otimes (\Delta W_m^{(M),0} - { E_M}\big[ \Delta W^{(M),0} \big]).
\end{eqnarray*}
The ensemble Kalman filter updates the ensembles by
\begin{eqnarray*}
 \Delta Q_m^{(M),0} \leftarrow \Delta Q_m^{(M),0} + \mathcal{T}_{qw}^{(M),0} (\mathcal{T}_{ww}^{(M),0} + \mathcal{I})^{-1} (u^\delta - \mathcal{K}_2 \mathcal{F} \mathcal{K}_1 q_0 - \Delta W_m^{(M),0}),
\end{eqnarray*}
for $m = 1,2,\cdots,M$.

\noindent\textbf{Solution at step $N_{\rm e}$:} $q^{(M),0} = q_0 + { E_M}\big[\Delta Q^{(M),0}\big]$

The operator $\mathcal{T}_{qw}^{(M),0} (\mathcal{T}_{ww}^{(M),0} + \mathcal{I})^{-1}$ is usually referred to as the Kalman gain operator (or matrix in the finite dimensional setting).
To see the relation between the ensemble Kalman filter and the Tikhonov-Phillips regularization, let
$$
\mathcal{C}^{(M),0} = \frac{1}{M} \sum_{m=1}^M (\Delta Q_m^{(M),0} - { E_M}\big[ \Delta Q^{(M),0} \big]) \otimes (\Delta Q_m^{(M),0} - { E_M}\big[ \Delta Q^{(M),0} \big]),
$$
then the following holds, cf. \cite[Proposition B.2]{parzer2022convergence}.
\begin{lemma} \label{lemma: TK EnKF equivalence linear}
Let $\mathcal{C}$ be a self-adjoint, semi-positive definite, injective, and trace class operator. Let $\widetilde{q}^{(M)}$ be the Tikhonov-Phillips regularized solution given by   \eqref{eqn: TP regularization  Delta qM} with regularization parameter $1/N_{\rm e}$, and  ${q}^{(M),0}$ be the ensemble Kalman filter solution with $N_{\rm e}$ iterations, then for any $p \in [1,\infty)$, it holds that
$$
\lim_{M\to \infty }E\big[ \|{q}^{(M),0} - \widetilde{q}^{(M)}  \|^p_X  \big]^{1/p} =0,
$$
and
$$
\lim_{M\to \infty }E\big[ \|\mathcal{C}^{(M),0} - \widetilde{\mathcal{C}}^{(M)}  \|^p_{\mathcal{L}(X;X)}  \big]^{1/p} =0.
$$
\end{lemma}
The above lemma indicates that the ensemble Kalman filter can be seen as a Tikhonov-Phillips regularization with a stochastic low-rank approximation $\mathcal{C}^{(M),0}$ of the trace class operator $\mathcal{C}$. We point out that an adaptive Kalman filter was recently proposed in \cite{parzer2022convergence} for linear inverse problems and a complete treatment for nonlinear inverse problems seems still lacking.

\subsection{Ensemble Kalman filter for the fully nonlinear inverse scattering problem} \label{section: EnKF nonlinear}

For the  inverse scattering problem, the goal is to apply the ensemble Kalman filter as a derivative-free method, i.e., without   explicitly calculating the linearized operator \rev{$\mathcal{L}_0$} by the Fr\'echet derivative. To heuristically illustrate this idea, we first assume that the fully nonlinear operator admits the following approximation near the initial guess
 $$
 \mathcal{K}_2 \mathcal{F} \mathcal{K}_1 q \approx \mathcal{K}_2 \mathcal{F} \mathcal{K}_1 q_0 + \mathcal{L}_0 \Delta q.
 $$
 Later on we will eliminate the explicit use of the linearized operator $\mathcal{L}_0$. We first introduce
 $$
 Q^{(M),0}_m = q_0 + \Delta Q^{(M),0}_m, \quad m = 1,2, \cdots, M,
 $$
 and
 $$
  W^{(M),0}_m = \mathcal{K}_2 \mathcal{F} \mathcal{K}_1 q_0 + \Delta W^{(M),0}_m, \quad m = 1,2, \cdots, M,
 $$
 where $\Delta W^{(M),0}_m = \mathcal{L}_0 \Delta Q^{(M),0}_m$.  One can immediately obtain the following.
 \begin{prop}
     \rev{The following relations hold}
     \begin{eqnarray*}
     W^{(M),0}_m - { E_M}\big[ W^{(M),0} \big] =    \Delta W^{(M),0}_m -  { E_M}\big[ \Delta W^{(M),0} \big],
     \end{eqnarray*}
     and
     \begin{eqnarray*}   Q^{(M),0}_m - { E_M}\big[ Q^{(M),0} \big] =    \Delta Q^{(M),0}_m -  { E_M}\big[ \Delta Q^{(M),0} \big].
     \end{eqnarray*}
 \end{prop}
 Therefore we can rewrite the ensemble Kalman filter with $\Delta Q^{(M),0}$ update by the   ensemble Kalman filter with $Q^{(M),0}$ update. In particular, given \rev{an} $X$-valued ensemble $Q^{(M),0}=[ Q^{(M),0}_1, Q^{(M),0}_2, \cdots,  Q^{(M),0}_M]$ where each $X$-valued random element $Q^{(M)\rev{,0}}_m$, $m=1,2,\cdots, M$, is  chosen according to mean $q_0$ and covariance $\mathcal{C}$, the ensemble Kalman filter in Section \ref{section: linear general EnKF} can be rewritten as follows.

\noindent\textbf{Iteration ($n\to n+1$ until $N_{\rm e}$):} Compute
$$
 W_m^{(M),0} = \mathcal{K}_2 \mathcal{F} \mathcal{K}_1 q_0 + \mathcal{L}_0 (Q_m^{(M),0} - q_0), \quad m=1,2,\cdots,M.
$$
Generate two operators $\mathcal{T}_{qw}^{(M),0}: X \to X$ and $\mathcal{T}_{ww}^{(M),0}: X \to X$ via
\begin{eqnarray*}
    \mathcal{T}_{qw}^{(M),0} &=& \frac{1}{M} \sum_{m=1}^M (Q_m^{(M),0} - { E_M}\big[ Q^{(M),0} \big]) \otimes ( W_m^{(M),0} - { E_M}\big[  W^{(M),0} \big]),  \\
       \mathcal{T}_{ww}^{(M),0} &=& \frac{1}{M} \sum_{m=1}^M (W_m^{(M),0} - { E_M}\big[  W^{(M),0} \big]) \otimes ( W_m^{(M),0} - { E_M}\big[ W^{(M),0} \big]).
\end{eqnarray*}
Then the ensemble Kalman filter updates the ensembles by
\begin{eqnarray*}
Q_m^{(M),0} \leftarrow Q_m^{(M),0} + \mathcal{T}_{qw}^{(M),0} (\mathcal{T}_{ww}^{(M),0} + \mathcal{I})^{-1} (u^\delta -  W_m^{(M),0}), \quad m = 1,2,\cdots,M,
\end{eqnarray*}

\noindent\textbf{Solution at step $N_{\rm e}$:} $q^{(M),0} =  { E_M}\big[Q^{(M),0}\big]$

The superscript $0$ represents that the algorithm still explicitly uses the linearized operator $\mathcal{L}_0$. To implicitly use the idea of linearization, note that
$$
 W_m^{(M),0} = \mathcal{K}_2 \mathcal{F} \mathcal{K}_1 q_0 + \mathcal{L}_0 (Q_m^{(M),0} - q_0) \approx \mathcal{K}_2 \mathcal{F} \mathcal{K}_1 Q_m^{(M),0}, \quad m=1,2,\cdots,M,
$$
one can then replace each $W_m^{(M),0}$ by $W_m^{(M)} = \mathcal{K}_2 \mathcal{F} \mathcal{K}_1 Q_m^{(M)}$ and drop all the superscript $0$ to arrive at Algorithm \ref{algo: EnKF nonlinear}, the ensemble Kalman filter for solving the inverse scattering problem. In Algorithm \ref{algo: EnKF nonlinear}, recall that $\mathcal{K}_1$ and $\mathcal{K}_2$ are defined via \eqref{eqn: def K1} and \eqref{eqn: def K2}, respectively, and $\mathcal{F}: L^2(B) \to L^2(B)$ is the reformulated forward operator (cf. Section \ref{section: model reformulation}).
\begin{algorithm}
	\caption{Ensemble Kalman filter for  inverse medium scattering}
	\begin{algorithmic}[1]
        \Require{Noisy data $u^\delta \in X$, initial guess $q_0 \in X$, iteration number $N_{\rm e}$, trace class operator $\mathcal{C}$, ensemble of $X$-valued random elements $Q^{(M)}=[Q^{(M)}_1, Q^{(M)}_2, \cdots, Q^{(M)}_M]$} with mean $q_0$ and covariance $\mathcal{C}$
         \Ensure{Solution $q^{(M)}$}
         \For{$n=1,2,\cdots,N_{\rm e}$}
         \State
    $W_m^{(M)} = \mathcal{K}_2 \mathcal{F} \mathcal{K}_1 Q_m^{(M)}, \quad m=1,2,\cdots,M$
    \State $w^{(M)} = { E_M}\big[   W^{(M)} \big]$, $ q^{(M)} = { E_M}\big[ Q^{(M)} \big]$
         \State $\mathcal{T}_{qw}^{(M)} = \frac{1}{M} \sum_{m=1}^M (Q_m^{(M)} - q^{(M)}) \otimes ( W_m^{(M)} -  w^{(M)}),$
         \State $\mathcal{T}_{ww}^{(M)} = \frac{1}{M} \sum_{m=1}^M (W_m^{(M)} - w^{(M)}) \otimes ( W_m^{(M)} -  w^{(M)})$
         \State $Q_m^{(M)} += \mathcal{T}_{qw}^{(M)} (\mathcal{T}_{ww}^{(M)} + \mathcal{I})^{-1} (u^\delta -  W_m^{(M)}), \quad m = 1,2,\cdots,M$
         \EndFor
         \State $q^{(M)} = { E_M}\big[Q^{(M)}\big]$
	\end{algorithmic}
 \label{algo: EnKF nonlinear}
\end{algorithm}

\subsection{Customized Sobolev space}
The covariance operator $\mathcal{C}: X \to X$ represents a priori knowledge about the unknown contrast. Motivated by the application of disk PSWFs to the inverse scattering problem \cite{meng23data}, we look for unknown contrasts in a customized Sobolev space and will introduce an appropriate   covariance operator later on.
With the help of Sturm-Liouville theory, one can define the following customized Sobolev space
\begin{equation*}
    H_{c}^{\tilde{s}}(B):=\{ u \in L^2(B): \sum^{l\in\mathbb{I}(m)}_{m,n\in \mathbb{N}}\chi^{\tilde{s}}_{m,n}|\langle u, \psi_{m,n,\ell} \rangle|^2 < \infty  \}
\end{equation*}
for any integer ${\tilde{s}}=1,2,\cdots$,
where $\langle u, \psi_{m,n,\ell} \rangle = \int_B u(x)  \psi_{m,n,\ell}(x) {\rm d}x$. By interpolation theory, the above Sobolev space $H_{c}^{\tilde{s}}(B)$ is well defined for any real number ${\tilde{s}}\ge0$.

\begin{remark} \label{Lemma chimn asymptotic}
We remark that the following property of the Sturm-Liouville eigenvalue $\chi_{m,n}$ \cite{ZLWZ20} holds: for any $c>0$,
\begin{eqnarray*}
(m+2n)(m+2n+2) < \chi_{m,n}(c) <  (m+2n)(m+2n+2)+c^2.
\end{eqnarray*}
We also remark that the customized Sobolev space $H_c^{\tilde{s}}(B)$ is practically useful since it is closely related to the standard Sobolev space $H^{\tilde{s}}(B)$ and more details can be found in \cite{meng23data} and the references therein.
\end{remark}
\subsection{Low-rank-assisted formulation}
The above customized  Sobolev space allows \rev{us} to introduce an appropriate covariance operator $\mathcal{C}: X \to X$ in the ensemble Kalman filter.
To elaborate the idea,
define the following operator $\mathcal{C}_0: L^2(B;\mathbb{R}) \to L^2(B;\mathbb{R})$ by
$$
\mathcal{C}_0 \psi_{m,n,\ell} = \vartheta \chi_{m,n}^{-s} \psi_{m,n,\ell}, \quad m,n \in \mathbb{N}, \ell \in \mathbb{I}(m),
$$
where $s>0$ is a positive real number such that $\mathcal{C}_0$ is of trace class, i.e.,
$$
\sum_{m,n\in \mathbb{N}}^{\ell \in \mathbb{I}(m)}\chi_{m,n}^{-s} < \infty,
$$
and $\vartheta>0$ is a positive scaling that heuristically represents the a priori knowledge about the amplitude level of the unknown.
It is directly seen that the above series   converges for $s>1$ due to the   asymptotic property (cf. Remark \ref{Lemma chimn asymptotic}) where
\begin{eqnarray*}
(m+2n)(m+2n+2) < \chi_{m,n}(c) <  (m+2n)(m+2n+2)+c^2.
\end{eqnarray*}
It also follows that $\mathcal{C}_0 =\vartheta \mathcal{D}^{-s}$. Now we introduce the covariance operator $\mathcal{C}: X \to X$ via
$$
\mathcal{C} = \left(\begin{matrix}
\mathcal{C}_0 & 0 \\
0 & \mathcal{C}_0
\end{matrix}\right)
$$

\begin{algorithm}
	\caption{Low-rank-assisted ensemble Kalman filter for inverse medium scattering}
	\begin{algorithmic}[1]
        \Require{Noisy processed data $u^\delta$, spectral cut off parameter $\eta$,   maximum iteration number $N_{\rm e}$}
         \Ensure{Solution $\sum_{(m,n,\ell) \in J_\eta} q^{(M)}_{m,n,\ell}   \psi_{m,n,\ell}$}
         \State Precompute the disk PSWFs system $\{\psi_{m,n,\ell}, \alpha_{m,n}\}$   and set $J_\eta = \{(m,n,\ell): |\alpha_{m,n}|>\eta \}$
         \State Compute $u^{\delta,\mbox{\footnotesize proj}} = \mathcal{P}_\eta u^\delta$ where   $  u^{\delta,{\rm proj}}_{m,n,\ell} =\langle u^\delta, \psi_{m,n,\ell}\rangle$ for $(m,n,\ell) \in J_\eta$ and the low-rank regularized initial guess  $\{q_{(0),m,n,\ell}\}_{(m,n,\ell) \in J_\eta }$ where $q_{(0),m,n,\ell} = u^{\delta,{\rm proj}}_{m,n,\ell}/\alpha_{m,n}$; identify \rev{the initial guess $q_0$} with $\{q_{(0),m,n,\ell}\}_{(m,n,\ell) \in J_\eta}$ and   $u^{\delta,{\rm proj}}$   with $\{u^{\delta,{\rm proj}}_{m,n,\ell}\}_{(m,n,\ell) \in J_\eta}$
         \State Generate initial ensemble $Q^{(M)} = [Q^{(M)}_1, Q^{(M)}_2, \cdots, Q^{(M)}_M]$ with each $Q_j^{(M)}$ identified as a real-valued vector of length $2 (\mbox{dim}J_\eta)$ by stacking the real and imaginary parts of $\big\{q^{[j]}_{m,n,\ell}\big\}_{(m,n,\ell) \in J_\eta}$
where $\rev{\big(\mbox{Re}\, q^{[j]}_{m,n,\ell}, \mbox{Im}\, q^{[j]}_{m,n,\ell}\big)^T} \sim \rev{\big(\mbox{Re}\, q_{(0),m,n,\ell}, \mbox{Im}\, q_{(0),m,n,\ell}\big)^T} +  \sqrt{\vartheta\chi_{m,n}^{-s}} (\xi^{[j]}_1,\xi^{[j]}_2)^T$
where $\xi^{[j]}_1 \sim \mathcal{N}(0,1)$,  $\xi^{[j]}_2 \sim \mathcal{N}(0,1)$, and $\vartheta>0$.
         \For{$\rev{n}=1,2,\cdots,N_{\rm e}$}
         \State $W^{(M)} = [W^{(M)}_1, W^{(M)}_2, \cdots, W^{(M)}_M]$ where
    $W_m^{(M)} = \mathcal{P}_\eta \mathcal{K}_2 \mathcal{F} \mathcal{K}_1 Q_m^{(M)}, \, m=1,2,\cdots,M$
    \State $w^{(M)} = { E_M}\big[   W^{(M)} \big]$
         \State $\mathcal{T}_{qw}^{(M)} = \frac{1}{M} \sum_{m=1}^M \mbox{kron}\big( Q_m^{(M)} - q^{(M)},  \rev{( W_m^{(M)} -  w^{(M)})^T} \big)$  \Comment{{ \small Using Matlab kron}}
                  \State $\mathcal{T}_{ww}^{(M)} = \frac{1}{M} \sum_{m=1}^M \mbox{kron}\big(W_m^{(M)} - w^{(M)},  \rev{( W_m^{(M)} -  w^{(M)})^T} \big)$\Comment{{ \small Using Matlab kron}}
         \State
         $Q_m^{(M)} +=  \mathcal{T}_{qw}^{(M)} (\mathcal{T}_{ww}^{(M)} + \rev{\gamma_n} \mathcal{I})^{-1} (u^{\delta,{\rm proj}} -  W_m^{(M)}), \quad m = 1,2,\cdots,M$; $\rev{\gamma_n}$ is a regularization parameter
         \State $ q^{(M)} = { E_M}\big[  Q^{(M)} \big]$
         \If{stopping \rev{criterion} is satisfied}
           stop
         \EndIf
         \EndFor
         \State Return $q^{(M)}$
	\end{algorithmic}
 \label{algo: low-rank EnKF nonlinear}
\end{algorithm}
Given the perturbed data $u^\delta$, Theorem \ref{lemma: Lipschitz stability in low-rankspace} suggests the inverse Born approximation
$$
q_0 \approx \sum_{m,n,\ell \in J_\eta} q_{(0),m,n,\ell}   \psi_{m,n,\ell}
$$
where
$q_{(0),m,n,\ell} = \langle u^{\delta}, \psi_{m,n,\ell} \rangle/\alpha_{m,n}$
and $J_\eta = \{(m,n,\ell): |\alpha_{m,n}|>\eta \}$ determines the dimension of the low-rank space using the cut off $\eta>0$.  Such a low-rank space is expected to mitigate the ill-posedness of the inverse scattering problem.
With the trace class operator $\mathcal{C}$, we can generate the initial ensemble $Q^{(M)} = [Q^{(M)}_1, Q^{(M)}_2, \cdots, Q^{(M)}_M]$ by the Karhunen–Lo\`eve expansion (cf. \cite{karhunen1946spektraltheorie,loeve1945fonctions} and \cite{ghanem2003stochastic})
$$
Q^{(M)}_j \sim  \sum_{m,n,\ell \in J_\eta} q^{[j]}_{m,n,\ell}   \psi_{m,n,\ell} \quad \mbox{for each} \quad j=1,2,\cdots,M,
$$
where
\begin{equation} \label{eqn: q^j_mnl KL expansion 1}
  \rev{\left(\begin{matrix} \mbox{Re}\, q^{[j]}_{m,n,\ell} \\ \mbox{Im}\, q^{[j]}_{m,n,\ell} \end{matrix}\right)} \sim \rev{\left(\begin{matrix} \mbox{Re}\, q_{(0),m,n,\ell} \\ \mbox{Im}\, q_{(0),m,n,\ell} \end{matrix}\right)} +   \sqrt{\vartheta\chi_{m,n}^{-s}}  \left(\begin{matrix} \xi^{[j]}_1 \\ \xi^{[j]}_2 \end{matrix}\right) ,\quad j=1,2,\cdots, M,
\end{equation}
where $\xi^{[j]}_1 \sim \mathcal{N}(0,1)$ and $\xi^{[j]}_2 \sim \mathcal{N}(0,1)$.
Similarly, the forward map leads to processed data $W_j^{(M)} = \mathcal{K}_2 \mathcal{F} \mathcal{K}_1 Q_j^{(M)}$ and we represent $W_j^{(M)}$  by its low-rank approximation
$$
W^{(M)}_j \approx \mathcal{P}_\eta  \mathcal{K}_2 \mathcal{F} \mathcal{K}_1 Q_j^{(M)} =  \sum_{m,n,\ell \in J_\eta} w^{[j]}_{m,n,\ell}   \psi_{m,n,\ell} \quad \mbox{for each} \quad j=1,2,\cdots,M.
$$
Now we arrive at the  low-rank-assisted formulation of the ensemble Kalman filter in Algorithm \ref{algo: low-rank EnKF nonlinear}, where we conveniently identify $Q_j^{(M)}$ as a real-valued vector of length $2 (\mbox{dim}J_\eta)$ by stacking the real and imaginary parts of $\big\{q^{[j]}_{m,n,\ell}\big\}_{m,n,\ell \in J_\eta}$ for each $j=1,2,\cdots,M$; same notational convenience applies to $W_j^{(M)}$, $j=1,2,\cdots,M$.

In Algorithm \ref{algo: low-rank EnKF nonlinear}, the parameter $\vartheta$ heuristically represents the a priori knowledge about the amplitude level of the unknown, and is fixed as $1$ in the numerical studies.
We heuristically introduce an additional regularization parameter $\rev{\gamma_n}$ in each iteration. In particular, one heuristic choice is $\rev{\gamma_n = \max\{0.1, \delta\} |\lambda_n|}$  where $\delta$ is the noise level and $\rev{\lambda_n}$ is the largest eigenvalue (in amplitude) of $\mathcal{T}_{ww}^{(M)}$ in the \rev{$n$-th} iteration; another conservative choice can be $\rev{\gamma_n = 0.9 |\lambda_n|}$, assuming no a priori information about the noise level; another choice is $\rev{\gamma_n=1}$ (as suggested in the first Algorithm \ref{algo: EnKF nonlinear}) which seems \rev{to require many} more iterations, and all   these three choices will be tested numerically.
It is known that an appropriate stopping \rev{criterion} is critical to ensure a good approximation in iterative algorithms such as the Gauss-Newton and Levenberg-Marquardt algorithms; however a complete analysis of the stopping criteria for the ensemble Kalman filter is difficult \cite{iglesias2013ensemble} and is beyond the scope of this work; heuristically, we point out that the ensemble Kalman filter can be chosen (cf. \cite{iglesias2013ensemble}) to terminate   for the first \rev{$n$} such that the residual in Frobenius norm $\big\|u^{\delta,{\rm proj}} - \mathcal{P}_\eta \mathcal{K}_2\mathcal{F}\mathcal{K}_1 ({ E_M}\big[Q^{(M)}\big])\big\|_{\rm F} \le  c_0\widetilde{\delta} $ for some $c_0>1$  where $\widetilde{\delta}$ represents the error due to modeling and noise. In a similar fashion, the ensemble Kalman filter can  stop at the first \rev{$n$} when the relative residual $\big\|u^{\delta,{\rm proj}} - \mathcal{P}_\eta \mathcal{K}_2\mathcal{F} \mathcal{K}_1 ({ E_M}\big[Q^{(M)}\big])\big\|_{\rm F}/\big\|u^{\delta,{\rm proj}} \|_{\rm F} < c_0\delta $ for some $c_0>1$ where $\delta$ represents the  noise level in the data; it is also possible to terminate when the relative residual   starts to stagnate. These stopping criteria will be discussed and illustrated in the numerical examples in the next section.

 \begin{figure}[htbp]
\includegraphics[trim=0cm 15cm 0cm 0cm, clip=true,width=1\textwidth]{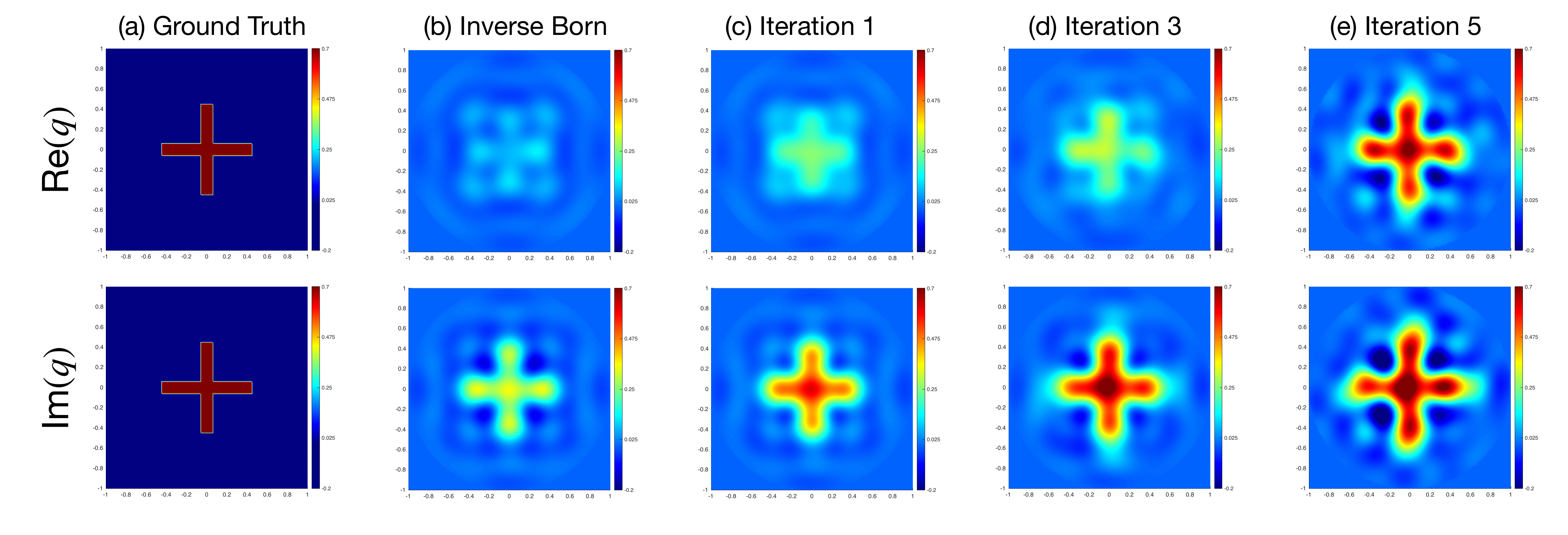}
\caption{Reconstruction of the strong scatterer ``Cross 2D''. Figure (a) plots the ground truth scatterer. Figure (b)(c)(d)(e) plot the reconstruction of the real part $\mbox{Re}(q)$ in the first row and the imaginary part $\mbox{Im}(q)$ in the second row at different iterations.
}
\label{figure cross2D}
\end{figure}

\section{Numerical experiments} \label{section: numerics}

In this section, we conduct numerical experiments to demonstrate the feasibility of the proposed low-rank-assisted ensemble Kalman filter. For the full model, we use IPscatt \cite{burgel2019algorithm} to generate the exact far-field data  $\{u^{\infty}(\hat{x}_i;\hat{\theta}_j;k)\}_{i,j=1}^{N}$ at equispaced incident and observation directions for $N=64$.   The noisy far field data $\{u^{\infty,\delta}(\hat{x}_i;\hat{\theta}_j;k)\}_{i,j=1}^{N}$ are generated by  adding   random uniformly distributed noise point-wise where the relative noise level is $\delta = 3\%$.  The noisy far field data $\{u^{\infty,\delta}(\hat{x}_i;\hat{\theta}_j;k)\}_{i,j=1}^{N}$ are \rev{further} processed according to Proposition \ref{prop: define u(p)} to obtain the processed data within the disk $B$ (cf. \cite{zhou2024exploring}). We also remark that the numerical discretization  in the forward solver for the far-field data generation is different from the one for the ensemble Kalman filter.
If the  degree of nonlinearity $\delta_1$ and the noise level $\delta$ are approximately known, one can compute the initial guess $q_0$ as in Theorem \ref{lemma: Lipschitz stability in low-rankspace}  using a  cut off parameter $\eta= \delta_2 |\alpha_{0,0}(c)|$ where $\delta_2 \in (0,1)$ represents the total effect of the degree of nonlinearity $\delta_1$ and the  noise level $\delta$. In practice the a \rev{priori} information on the degree of nonlinearity and noise level may not be  known, in this work the spectral cut off is chosen  conservatively as $\eta=0.9 |\alpha_{0,0}(c)|$ to facilitate robustness (cf. \cite{zhou2024exploring}).

The covariance operator is motivated by $\mathcal{D}^{-s}$   where we  generate the initial ensemble $Q^{(M)} = [Q^{(M)}_1, Q^{(M)}_2, \cdots, Q^{(M)}_M]$ by the Karhunen–Lo\`eve expansion
$$
Q^{(M)}_j \sim  \sum_{m,n,\ell \in J_\eta} q^{[j]}_{m,n,\ell}   \psi_{m,n,\ell} \quad \mbox{for each} \quad j=1,2,\cdots,M,
$$
where \rev{the real and imaginary parts of the coefficients are sampled as in \eqref{eqn: q^j_mnl KL expansion 1} with $\sqrt{\vartheta\chi_{m,n}^{-s}}$ replaced by $(m+2n+2)^{-s}$, i.e.,}
$$
\rev{\left(\begin{matrix} \mbox{Re}\, q^{[j]}_{m,n,\ell} \\ \mbox{Im}\, q^{[j]}_{m,n,\ell} \end{matrix}\right)} \sim \rev{\left(\begin{matrix} \mbox{Re}\, q_{(0),m,n,\ell} \\ \mbox{Im}\, q_{(0),m,n,\ell} \end{matrix}\right)} +    (m+2n+2)^{-s}  \left(\begin{matrix} \xi^{[j]}_1 \\ \xi^{[j]}_2 \end{matrix}\right)
$$
with $s=2.5$, $\xi^{[j]}_1 \sim \mathcal{N}(0,1)$ and $\xi^{[j]}_2 \sim \mathcal{N}(0,1)$, for each $j=1,2,\cdots, M$; to test other choices of $s$, we present the case when $s=1.5$ in Section \ref{section numeric regu parameter and stopping criteria}. Here we have replaced $\chi_{m,n}$ in equation \eqref{eqn: q^j_mnl KL expansion 1} by $(m+2n+2)^2$ thanks to the asymptotic in Remark \ref{Lemma chimn asymptotic}.

\subsection{Improving the inverse Born solution via iteration} \label{section numeric improving cross2d}

We first demonstrate how the inverse Born solution can be improved in several steps using the low-rank-assisted ensemble Kalman filter. The wave number is $k=10$.  The ground truth of the strong scatterer is the complex-valued \rev{``Cross 2D''} and we plot its real and imaginary parts in Figure \ref{figure cross2D}(a). The inverse Born solution in Figure \ref{figure cross2D}(b) gives poor approximations, and is improved in several iterations by the low-rank-assisted EnKF  with ensemble size   $100$, cf. Figure \ref{figure cross2D}(c)(d)(e). Specifically, it is observed that the amplitude is largely corrected in iteration $3$, and more details are added in \rev{iteration} $5$. We omit the plots in later  iterations since the resolution remains on the same level. Here  the regularization parameter in each iteration is  $\rev{\gamma_n = \max\{0.1, \delta\} |\lambda_n|}$  where $\delta$ is the noise level and $\rev{\lambda_n}$ is the largest eigenvalue in amplitude of $\mathcal{T}_{ww}^{(M)}$ in the \rev{$n$-th} iteration. The relative residual history is plotted in Figure \ref{figure cross2D residual}(a) with marker circle. It is observed that the relative residual  $\big\|u^{\delta,{\rm proj}} - \mathcal{P}_\eta \mathcal{K}_2 \mathcal{F} \mathcal{K}_1 ({ E_M}\big[Q^{(M)}\big])\big\|_{\rm F}/\big\|u^{\delta,{\rm proj}} \|_{\rm F}$ decays fast in the first few iterations and then suddenly stagnates, therefore it is reasonable for the ensemble Kalman filter to terminate at iteration step $5$; moreover,
 due to the fast decay of the relative residuals, it is also reasonable to terminate the ensemble Kalman filter  when the relative residual  $\big\|u^{\delta,{\rm proj}} - \mathcal{P}_\eta \mathcal{K}_2 \mathcal{F} \mathcal{K}_1 ({ E_M}\big[Q^{(M)}\big])\big\|_{\rm F}/\big\|u^{\delta,{\rm proj}} \|_{\rm F} < c_0\delta $ for some constant $c_0$.  In the following, we will illustrate the stopping criteria in more numerical examples.

\subsection{Test on different choices of regularization parameters} \label{section numeric regu parameter and stopping criteria}
  \begin{figure}[htbp]
\includegraphics[trim=0cm 15cm 0cm 0cm, clip=true,width=1\textwidth]{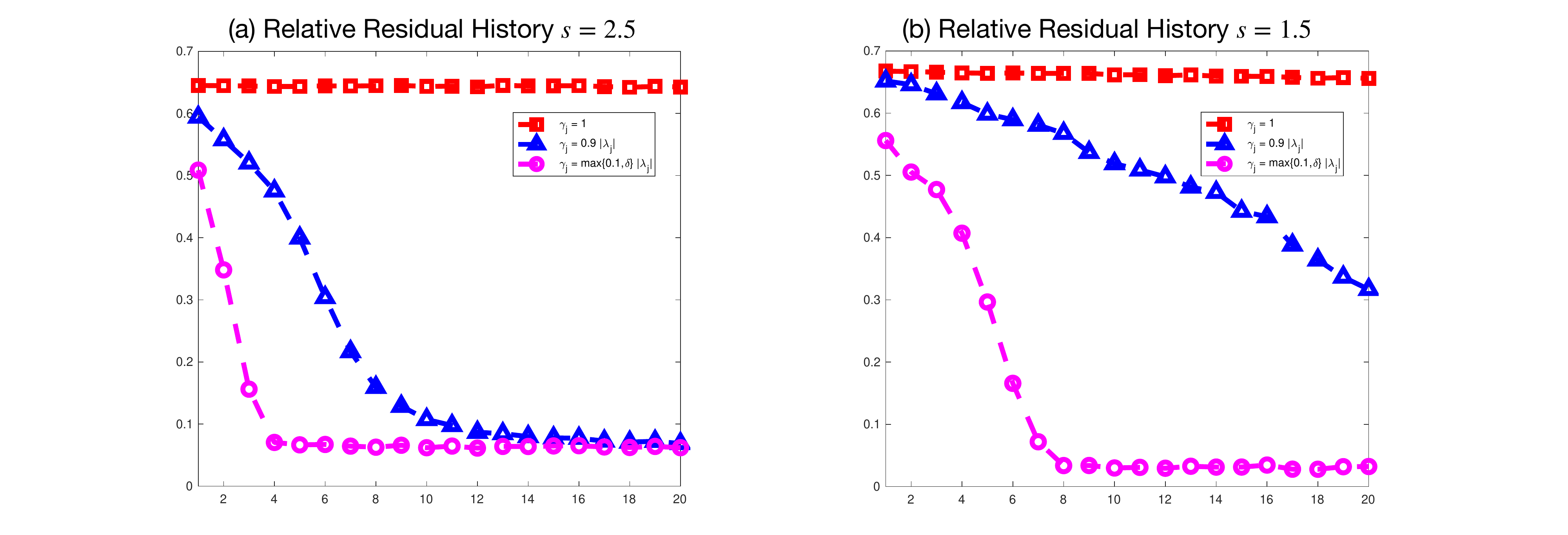}
\caption{Relative residual history of the ensemble Kalman filter for the strong scatterer ``Cross 2D''. \rev{Figures} (a) and (b) \rev{plot} the history for $s=2.5$ and $s=1.5$, respectively.
}
\label{figure cross2D residual}
\end{figure}
To test the convergence for different choices of the regularization parameters $\rev{\gamma_n}$, we plot the history of the relative residual with $N_{\rm e}=20$ in Figure \ref{figure cross2D residual}(a).
We first plot the relative residual history    when the regularization parameter is fixed as $\rev{\gamma_n=1}$  in Figure \ref{figure cross2D residual}(a) with marker square and it is observed that the relative residual decreases very slowly. In the case when $\rev{\gamma_n= 0.9 |\lambda_n|}$  where $\rev{|\lambda_n|}$ is the largest eigenvalue in amplitude of $\mathcal{T}_{ww}^{(M)}$  at each iteration step \rev{$n$}, it is observed that the convergence is much faster (cf. Figure \ref{figure cross2D residual}(a) with marker triangle). Finally we test the case when $\rev{\gamma_n= \max\{0.1, \delta\} |\lambda_n|}$ where $\delta=0.03$ represents the noise level, the convergence (cf. Figure \ref{figure cross2D residual}(a) with marker circle) is slightly faster than the case when $\rev{\gamma_n= 0.9 |\lambda_n|}$.   To further test other choices of $s$ in Algorithm \ref{algo: low-rank EnKF nonlinear} (i.e. Line 3), we plot in Figure \ref{figure cross2D residual}(b) the relative residual history when $s=1.5$ and \rev{a similar behavior can be observed again}. In later examples, we  fix $s=2.5$ and $\rev{\gamma_n = \max\{0.1, \delta\} |\lambda_n|}$.

\begin{figure}[h]
\includegraphics[trim=0cm 24cm 0cm 0cm, clip=true,width=1\textwidth]{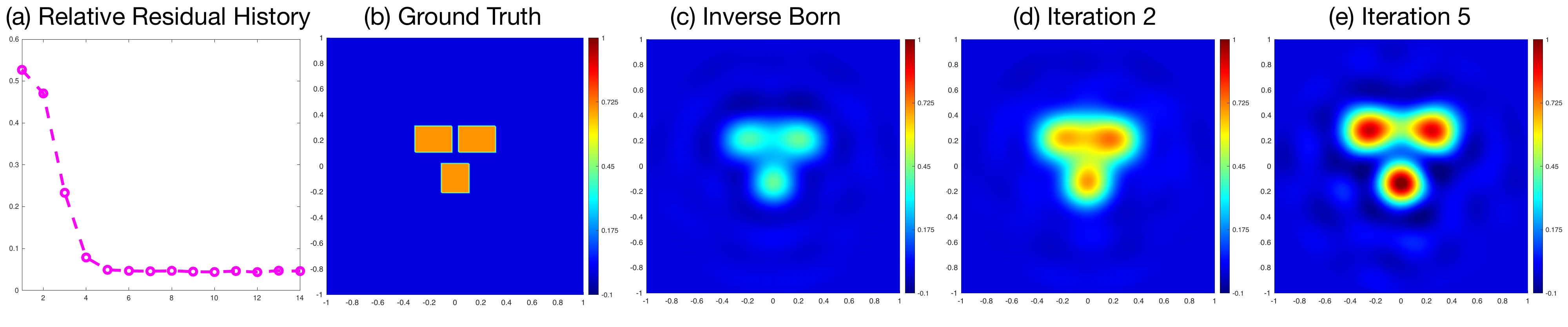}
\caption{Reconstruction of three rectangles at wave number $k=10$. Figure(a) plots the relative residual history. To the right, Figure(b) plots the real-valued  ground truth  contrast, and Figure(c)(d)(e) plots the reconstructions at different iteration steps.
}
\label{figure ThreeSquares}
\end{figure}

\subsection{Nearby objects and different wave numbers} \label{section: numerical nearby}

\begin{figure}[h]
\includegraphics[trim=0cm 24cm 0cm 0cm, clip=true,width=1\textwidth]{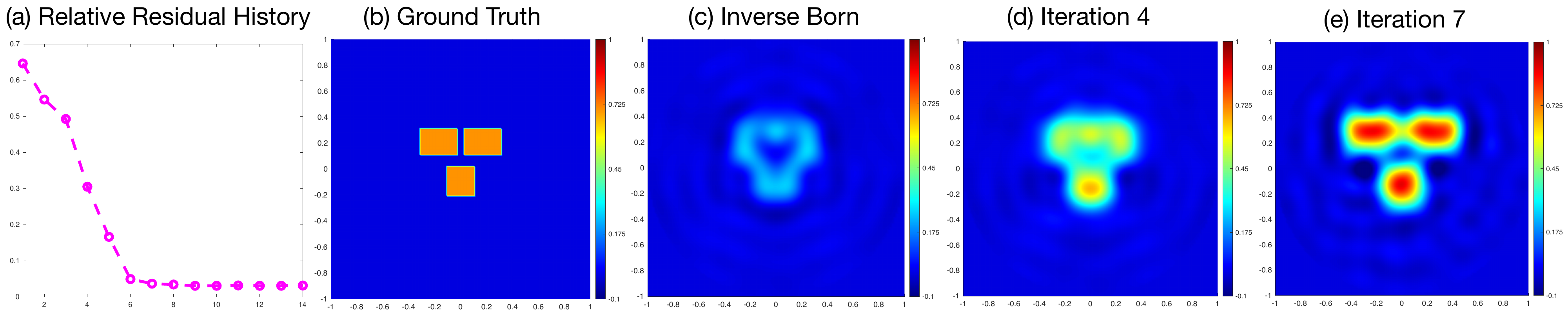}
\caption{Reconstruction of three rectangles at wave number $k=15$.   Figure(a) plots the relative residual history. To the right, Figure(b) plots the real-valued  ground truth contrast, and Figure(c)(d)(e) plots the reconstructions at different iteration steps.
}
\label{figure ThreeSquares k15}
\end{figure}
In \rev{Figures} \ref{figure ThreeSquares}--\ref{figure ThreeSquares k15}, we further test the reconstruction of three nearby rectangles where the smallest distance between the rectangles is $0.05$. In these examples the contrast is real-valued, so that we can verify numerically that the proposed   method also works in the real-valued case. At wave number $k=10$, the low-rank regularized inverse Born approximation in Figure \ref{figure ThreeSquares}(c) fails to distinguish the top two rectangles  and the amplitude is  far off from the ground truth. The amplitude is again largely corrected  in the next couple of iterations  and more details are added during later iterations. In iteration $5$, the three rectangles become \rev{separated}.  The relative residual history is plotted in Figure \ref{figure ThreeSquares}(a) which verifies the stopping criteria discussed in Section \rev{\ref{section numeric improving cross2d}}.

We further increase the wave number to $k=15$ in Figure \ref{figure ThreeSquares k15} to test whether the resolution can be improved or not. The larger the wave number, the more the degree of nonlinearity; as a result, the inverse Born reconstruction in Figure \ref{figure ThreeSquares k15}(c) using the linearized low-rank structure becomes worse. The amplitude of the unknown contrast is significantly corrected after a few iterations, and the three rectangles become separated gradually. The relative residual history is plotted in Figure \ref{figure ThreeSquares k15}(a). We also observe that the stopping criteria discussed in Section \rev{\ref{section numeric improving cross2d}} \rev{are} also applicable to this numerical test. Furthermore, we point out that the dimension of the low-rank space increases as the wave number becomes larger, and the ensemble size is only on the order of the dimension of the proposed low-rank space.
Finally the resolution in Figure \ref{figure ThreeSquares k15}(e) with $k=15$ is better than the resolution in Figure \ref{figure ThreeSquares}(e) with $k=10$, and similar improved resolution (i.e. increasing stability) was also observed  in the linearized region \cite{zhou2024exploring} and \cite{isaev2022numerical} for weak scatterers.

\section{\rev{Conclusion}}
In this work we propose to use the low-rank structure to solve the inverse medium scattering problem beyond the Born or linearized region. The proposed low-rank space is intrinsic to the linearized forward map and its dimension is intrinsically determined by the wave number, in contrast to some other heuristic ways of choosing low-rank spaces. In the proposed low-rank space, our first contribution is to establish the Lipschitz stability in the fully nonlinear case and characterize the explicit Lipschitz constant in the linearized region. Our second contribution is to propose a low-rank-assisted ensemble Kalman filter to reconstruct the contrast numerically, where the solutions are updated iteratively in the proposed low-rank space and a new covariance operator  is proposed according to the connection between the low-rank structure and a Sturm-Liouville differential operator. Numerical examples are further provided to illustrate the feasibility of the low-rank-assisted method.
Looking ahead, it is interesting to integrate the low-rank structure with data-driven approaches that exploit a priori information of the unknown provided it is available.

\bibliographystyle{siamplain}

\end{document}